\documentclass{amsart}  
\usepackage{lipsum}
\usepackage{amsmath,amsbsy,amsfonts,latexsym}
\usepackage{color}
\usepackage[utf8]{inputenc}
\usepackage{graphicx}

\numberwithin{equation}{section}


\newtheorem{theorem}{Theorem}
\newtheorem{proposition}{Proposition}[section]
\newtheorem{corollary}[proposition]{Corollary}
\newtheorem{definition}[proposition]{Definition}
\newtheorem{lemma}[proposition]{Lemma}
\newtheorem{remark}[proposition]{Remark}

\newcommand{\diver}{\operatorname{div}}
\newcommand{\dd}{\mathrm{d}}

\begin{document}

\title[Pohozaev-type identities]{Pohozaev-type identities for a pseudo-relativistic Schr\"odinger operator and applications}
	\thanks{H. Bueno is the corresponding author; G.A. Pereira  received research grants by PNPD/CAPES/Brazil; A.H. Souza Medeiros by CNPq/Brazil.}

\author{H. Bueno}
\address{all authors - Departmento de Matem\'atica, Universidade Federal de  Minas Gerais, 31270-901 - Belo Horizonte - MG, Brazil}
\email{\begin{align*}H.\ Bueno:\ hamilton.pb@gmail.com&\\ 
	G.A.\ Pereira:\ gilbertoapereira@yahoo.com.br&\\A.H.\ Souza \ Medeiros:\	aldomedeiros@ufmg.br&\end{align*}}
\author{G.A. Pereira}
\author{A.H. Souza Medeiros}


\begin{abstract}
In this paper we prove a Pohozaev-type identity for both the problem $(-\Delta+m^2)^su=f(u)$ in $\mathbb{R}^N$ and its harmonic extension to $\mathbb{R}^{N+1}_+$ when $0<s<1$. So, our setting includes the pseudo-relativistic operator $\sqrt{-\Delta+m^2}$ and the results showed here are original, to the best of our knowledge. The identity is first obtained in the extension setting and then ``translated" into the original problem. In order to do that, we develop a specific Fourier transform theory for the fractionary operator $(-\Delta+m^2)^s$, which lead us to define a weak solution $u$ of the original problem if the identity
\begin{equation}\label{defsola}\int_{\mathbb{R}^N}(-\Delta+m^2)^{s/2}u(-\Delta+m^2)^{s/2}v\dd x=\int_{ \mathbb{R}^N}f(u)v\dd x\tag{S}\end{equation} is satisfied by all $v\in H^{s}(\mathbb{R}^N)$. The obtained Pohozaev-type identity is then applied to prove both a result of nonexistence of solution to the case $f(u)=|u|^{p-2}u$ if $p\geq 2^{*}_s$ and a result of existence of a ground state, if $f$ is modeled by $\kappa u^3/(1+u^2)$, for a constant $\kappa$. In this last case, we apply the Nehari-Pohozaev manifold introduced by D. Ruiz. Finally, we prove that positive solutions of $(-\Delta+m^2)^su=f(u)$ are radially symmetric and decreasing with respect to the origin, if $f$ is modeled by functions like $t^\alpha$, $\alpha\in(1,2^{*}_s-1)$ or $t\ln t$.
\end{abstract}
\maketitle
\tableofcontents
\section{Introduction}\label{intro}
The pseudo-relativistic Schr\"odinger operator $\sqrt{-\Delta+m^2}$ is associated with the Hamiltonian $\mathcal{H}=\sqrt{p^2c^2+m^2c^4}$ of a free relativistic particle of mass $m$ by the usual quantization $p\to i\hbar\Delta$, changing units so that $\hbar=1$ and $c=1$. A good exposition of the basic properties of the operator $\sqrt{-\Delta+m^2}$ can be found in \cite{LiebLoss}, see also \cite{Stein}. 

If $m>0$ denotes the mass of bosons in units, the equation
\begin{equation}\label{original1}\left\{
\begin{array}{rcll}
i \partial_t \psi &=&\left(\sqrt{-\Delta+m^2}-m\right)\psi+ \left(W*|\psi|^2\right)\psi &\text{in }\ \mathbb{R}^{N},\\
u(x,0)&=&\phi(x),  &x\in \mathbb{R}^{N} \end{array} \right.
\end{equation}
where $N\geq 2$ and $*$ denotes convolution, was used to describe the dynamics of pseudo-relativistic boson stars in astrophysics. See \cite{Cho,CSS,Elgart,Lieb2} for more details. The existence of solitary waves solutions for \eqref{original1}
\[\psi(x)=e^{i\mu t}\varphi(x)\]
with $\varphi$ satisfying the following pseudo-relativistic Hartree equation
\begin{equation}\label{Belchior}\sqrt{-\Delta+m^2}\varphi-m\varphi-(W*|\varphi|^2)\varphi=-\mu\varphi\end{equation}
was first established by Lieb and Yau \cite{Lieb2}, in the case $W(x)=|x|^{-1}$. 

For generalizations or variations on \eqref{Belchior}, the existence of  ground state solutions was obtained by various authors \cite{BBMP,Cingolani,CSS,ZelatiNolasco,ZelatiNolasco2,MaZhao,Moroz1}. A good survey on equations like \eqref{Belchior} is given in \cite{Moroz2}.

\noindent \textbf{Comparison between the operators $(-\Delta)^s$ and $(-\Delta+m^2)^s$.}
At first sight, one supposes that the treatment of both operators might be similar. In fact, there are huge differences between them. 
\begin{enumerate}
\item [(a)] $(-\Delta)^s$ is $2s$-homogeneous with respect to dilatations, that is, $(-\Delta)^s u(\lambda x)=\lambda^{2s}(-\Delta)^su(x)$, while such a property is not valid for $(-\Delta+m^2)^s$.
\item [(b)] As will see, $(-\Delta+m^2)^s$ generates a norm in $H^s(\mathbb{R}^N)$ and this is not the case for $(-\Delta)^s$. In consequence, the adequate spaces to handle both operators are quite different. 
\item [(c)] Some results about fractionary Laplacian spaces are now standard (see \cite{Demengel,Guide,Bisci}), but not so easy to find for $(-\Delta+m^2)^s$. See, however, \cite{Ambrosio,VFelli}.
\end{enumerate}

\noindent \textbf{Why to handle $(-\Delta+m^2)^s$ instead of $\sqrt{-\Delta+m^2}$.}

In this paper we deal with a generalized version of the operator $\sqrt{-\Delta+m^2}$, namely the operator $T(u)=(-\Delta+m^2)^su$, $0<s<1$. We study the problem \begin{equation}\label{original}(-\Delta+m^2)^su=f(u),\quad x\in \mathbb{R}^N.\end{equation}

Concerning the applications of equation \eqref{original}, we recall that fractionary Laplacian operators are the infinitesimal generators of L\'evy stable diffusion processes. In particular, $(-\Delta+m^2)^s-m^{2s}$ is called the $2s$-stable relativistic process, see \cite{Ambrosio,Carmona,Ryznar}. Stable diffusion processes have application in several areas such as anomalous diffusion of plasmas, probability, finances and populations dynamics, see \cite{Applebaum}.

Our approach applies the Dirichlet-to-Neumann operator, that is, we consider the extension problem naturally related to \eqref{original} for the operator $(-\Delta+m^2)^s$, thus resting on the celebrated papers by Cabr\'e and Sol\`a-Morales \cite{Cabre} and Caffarelli and Silvestre \cite{Caffarelli}.  

We state a general result about the extension problem:

\textbf{Theorem} (Stinga-Torrea \cite{StingaTorrea})\textit{ Let $h\in Dom(L^s)$ and $\Omega$ be an open subset of $\mathbb{R}^N$. A solution of the extension problem
	\[\left\{\begin{array}{ll}
	-L_x u+\frac{1-2s}{y}u_y+u_{yy}=0 &\text{in }\ \Omega\times (0,\infty)\\
	u(x,0)=h(x) &\text{on }\ \Omega\times \{y=0\}
	\end{array}\right.\]
	is given by
	\[u(x,y)=\frac{1}{\Gamma(s)}\int_{0}^\infty e^{-tL}(L^s h)(x)e^{-\frac{y^2}{4t}}\frac{\dd t}{t^{1-s}}\]
	and satisfies
	\[\lim_{y\to 0}y^{1-2s}u_y(x,y)=\frac{2s\Gamma(-s)}{4^s\Gamma(s)}(L^s h)(x).\]}\vspace*{.2cm}

Applying this result to \eqref{problem}, the extension problem produces, for $(x,y)\in \mathbb{R}^N\times (0,\infty)=\mathbb{R}^{N+1}_+$,
\begin{equation}\left\{\begin{array}{ll}
\displaystyle\Delta_x w+\frac{1-2s}{y}w_y+w_{yy}-m^2w=0 &\text{in }\ \mathbb{R}^{N+1}_+\vspace{.1cm} \\
\displaystyle\lim_{y\to 0^+}\left(-y^{1-2s}\frac{\partial w}{\partial y}\right)=k_s f(w(x,0)), &\text{in }\ \mathbb{R}^N\times \left\{ 0\right\} \simeq\mathbb{R}^N,
\end{array}\right.\tag{$P$}\label{P}
\end{equation}
where \begin{equation}\label{ks}
k_s=\frac{2s\Gamma(-s)}{4^s\Gamma(s)}=\frac{2^{1-2s}\Gamma(1-s)}{\Gamma(s)}.
\end{equation}

Applying Fourier transforms with respect to $x\in\mathbb{R}^N$, we are lead to the Bessel function attached to this problem, which is trivial in the case of $s=1/2$, but not for an arbitrary $s\in (0,1)$. (See discussion below.)

The natural setting for problem \eqref{P} is the Hilbert space
\[H^1(\mathbb{R}^{N+1}_+,y^{1-2s})=\left\{u\colon \mathbb{R}^{N+1}_+\to\mathbb{R}: \iint_{\mathbb{R}^{N+1}_+}\left(|\nabla w|^2+|w|^2\right)y^{1-2s}\dd x\dd y<\infty \right\}\]
endowed with the norm
\[\|w\|_s=\left(\iint_{\mathbb{R}^{N+1}_+}\left(|\nabla w(x,y)|^2+|w(x,y)|^2\right)y^{1-2s}\dd x\dd y\right)^{\frac{1}{2}},\]
see \cite{BMP} for details. 

Observe that, in the case $s=1/2$, the problem \eqref{P} is set in a much simpler space, since the weight $y^{1-2s}$ does not appear in the definition of $H^1(\mathbb{R}^{N+1}_+)$. We will return to this difference later on.

\noindent\textbf{The definition of solution.}
Although the definition of solution for \eqref{defsola} is easily obtainable in terms of  Fourier transforms as in M. M. Fall and V. Felli \cite{VFelli}, we do think that our definition is more natural, since it deals with a partial integration and remembers the one used in fractionary spaces.

This lead us to develop a specific Fourier theory for the fractionary operator $(-\Delta+m^2)^s$ in the case $0<s<1$, with an approach influenced  by Stinga and Torrea \cite{StingaTorrea} and also by Stein \cite{Stein}.

Therefore, we justify the definition  \begin{equation*}(-\Delta+m^2)^s  f(x)=\mathcal{F}^{-1}\left((m^2+4\pi^2|\cdot|^2)^s  \hat{f}\right)(x),\end{equation*}
see Section \ref{prelim}.

Proceeding with our analysis, we show that $(-\Delta+m^2)^s$ is symmetric in the space $\mathcal{S}(\mathbb{R}^N)$ and we finally achieve the natural definition of a weak solution to the problem \eqref{original}.
\begin{definition}\label{defsolution}
	A function $u\in H^{s}(\mathbb{R}^N)$ is a solution of \eqref{prn} if
	\begin{equation}\label{innerproduct}\int_{ \mathbb{R}^N}(-\Delta+m^2)^{s/2}u(-\Delta+m^2)^{s/2}v\dd x=\int_{ \mathbb{R}^N}f(u)v\dd x\end{equation}
	for all $v\in H^{s}(\mathbb{R}^N)$.
\end{definition}

The space $H^s(\mathbb{R}^N)$ will be considered with the norm generated by the inner product defined by the left-hand side of \eqref{innerproduct}:
\[\|u\|^2=\int_{ \mathbb{R}^N}\left|(m^2-\Delta)^{s/2}u(x)\right|^2\dd x=\int_{\mathbb{R}^N}(m^2+4\pi^2|\xi|^2)^s|\hat{u}(\xi)|^2\dd\xi.
\] 

\noindent\textbf{A Pohozaev-type identity for problem \eqref{P}.}

After that, we establish a Pohozaev-type identity for the extension problem. With additional (but natural) hypotheses, it is not difficult to consider $f(x,u)$ instead of $f(u)$. We recall that a local approach, based on integration in balls $B^+_r\subset \mathbb{R}^{N+1}_+$, was obtained in \cite{VFelli}.

The Pohozaev-type identity is a valuable tool when proving results of non-existence of non-trivial solutions for non-linear problems. It is also associated with the Pohozaev manifold generated by this identity, which is a precious technique in solving problems when either the (PS)-condition or the mountain pass geometry are difficult to be verified, see \cite{LehrerMaia,Ruiz}.

Although arguments leading to Pohozaev-type identities are beginning to be standard, we present its proof in the case of the extension problem \eqref{P} (see by X. Chang and Z-Q. Wang \cite{ChangWang}). 
\begin{multline*}
\frac{N-2s}{2}\iint_{ \mathbb{R}^{N+1}_+}y^{1-2s}|\nabla w|^2\dd x\dd y+m^2\frac{N+2-2s}{2}\iint_{\mathbb{R}^{N+1}_+}y^{1-2s}w^2\dd x\dd y\hfill\end{multline*}
\vspace*{-.3cm}\begin{align}
&=Nk_s\int_{ \mathbb{R}^{N+1}_+}F(w(x,0))\dd x\dd y,\label{Pohoextintrod}
\end{align}
where $F(t) = \int^t_0 f(\tau) \dd\tau$.
 
Observe the constant $k_s$ in this equation. Usually, changing scales, the constant $k_s$ is assumed to be equal to $1$. We decide not to do so in order to better understand how this constant appears implicitly in the left-hand side of the equation above.

\noindent\textbf{A Pohozaev-type identity for problem \eqref{original}.} Our next step was to ``translate"  the Pohozaev-type identity in $\mathbb{R}^{N+1}_+$ into the original setting in $\mathbb{R}^N$. The Fourier transform is our main technique, following Br\"andle, Colorado, de Pablo and S\'anchez \cite{Brandle}. This ``reinterpretation" is as consequence of the Fourier transform theory developed for the operator $(-\Delta+m^2)^s$ and the study of the Bessel function attached to the extension problem.


To interpret the last integral in \eqref{Pohoextintrod} as a integral in $\mathbb{R}^N$, we observe that a solution of problem \eqref{P} satisfies
\begin{equation}\label{initalproblem}\left\{\begin{array}{ll}
\displaystyle\Delta_x w+\frac{1-2s}{y}w_y+w_{yy}-m^2w=0 &\text{in }\ \mathbb{R}^{N+1}_+\vspace{.1cm} \\
w(x,0)=u(x), &x\in \mathbb{R}^{N},\end{array}\right.
\end{equation}
since problem \eqref{P} is a extension of problem \eqref{original}. In particular, $F(w(x,0))=F(u)$ and the right-hand side of \eqref{Pohoextintrod} causes no problem.

To proceed with the translation, we take the Fourier transform (in the variable $x\in\mathbb{R}^N$) of problem \eqref{initalproblem}. Its solution is given in terms of the Bessel function $\Phi_s$ by
\begin{equation}\label{solFourier}\hat{w}(\xi,y)=\hat{u}(\xi)\Phi_s\left(\sqrt{m^2+4\pi^2|\xi|^2}\,y\right),\end{equation}
where $\Phi_s$ solves (see \cite{Ambrosio,Brandle,BMP,Capella,VFelli})
\begin{equation}\label{Phi}-\Phi+\frac{1-2s}{y}\Phi'+\Phi''=0,\qquad \Phi(0)=1,\quad \lim_{y\to\infty}\Phi(y)=0.
\end{equation}

The asymptotic behavior satisfied by $\Phi_s$ is well-known:
\begin{equation}\label{asymp}\Phi_{s}(s)\sim \left\{\begin{array}{ll}
1-c_1y^{2s}, &\text{when }\ y\to 0,\\
c_2y^{(2s-1)/2}e^{-y}, &\text{when }\ y\to \infty,\end{array}\right.\end{equation}
where 
\begin{equation}\label{constants}c_1(s)=2^{1-2s}\frac{\Gamma(1-s)}{2s\Gamma(s)}\quad\text{and}\quad c_2(s)=\frac{2^{(1-s)/2}\pi^{1/2}}{\Gamma(s/2)},\end{equation}
(see \cite{Brandle,Capella}) and $\Phi_s$ is a minimum of the functional
\begin{equation}\label{KPhi}\mathcal{K}(\Phi)=\int_0^\infty\left(|\Phi(y)|^2+|\Phi'(y)|^2\right)y^{1-2s}\dd y.\end{equation}

In \cite{Brandle} is stated that $\mathcal{K}(\Phi_s)=k_s$ and that this value can be obtained applying integration by parts. We were not able to do so. The method we could find to prove that $\mathcal{K}(\Phi_s)=k_s$ was tricky. 

Observe that, in the case $s=1/2$, the Bessel function $\Phi_{1/2}(t)=e^{-t}$ is easy to manipulate. This fact motivates our exposition in the general setting $(-\Delta+m^2)^su$.

To translate the left-hand side of \eqref{Pohoextintrod} into $\mathbb{R}^N$, we first note that it can be written as
\[\frac{N-2s}{2}\iint_{ \mathbb{R}^{N+1}_+}y^{1-2s}\left[|\nabla w|^2+m^2w^2\right]\dd x\dd y+m^2\iint_{\mathbb{R}^{N+1}_+}y^{1-2s}w^2\dd x\dd y.\]

We show that 
\[\iint_{ \mathbb{R}^{N+1}_+}y^{1-2s}\left[|\nabla w|^2+m^2w^2\right]\dd x\dd y=k_s\int_{ \mathbb{R}^N}\left|(m^2-\Delta)^{s/2}u(x)\right|^2\dd x,\]
thus obtaining 
\begin{multline*}\frac{N-2s}{2}k_s\int_{ \mathbb{R}^N}\left|(m^2-\Delta)^{s/2}u(x)\right|^2\dd x+m^2\iint_{\mathbb{R}^{N+1}_+}y^{1-2s}w^2\dd x\dd y\hfill\end{multline*}
\vspace*{-.2cm}\begin{align*}
&=Nk_s\int_{ \mathbb{R}^N}F(u)\dd x
\end{align*}
and the constant $k_s$ already appears in the first integral of the left-hand side of the above equation. 

By applying Plancherel's identity and changing variables, we obtain
\begin{align*}m^2\iint_{\mathbb{R}^{N+1}_+}y^{1-2s}w^2\dd x\dd y&=m^2\int_{ \mathbb{R}^N}\frac{|\hat{u}(\xi)|^2\dd\xi}{\left(m^2+4\pi^2|\xi|^2\right)^{1-s}}\int_0^\infty [\Phi_s(t)]^2t^{1-2s}\dd t.
\end{align*}

Our last obstacle in the ``translation" was the evaluation of the integral en $\Phi_s$: 
\[\int_0^\infty [\Phi_s(t)]^2t^{1-2s}\dd t=sk_s.\]

This evaluation was obtained as a consequence of $\mathcal{K}(\Phi_s)=k_s$. As expected, the Pohozaev-type identity does not depend on $k_s$:
\begin{theorem}\label{1}A solution $u\in H^s(\mathbb{R}^N)$ of problem \eqref{original} satisfies
	\begin{multline*}\frac{N-2s}{2}\int_{ \mathbb{R}^N}\left|(m^2-\Delta)^{s/2}u(x)\right|^2\dd x+sm^2\int_{ \mathbb{R}^N}\frac{|\hat{u}(\xi)|^2\dd\xi}{\left(m^2+4\pi^2|\xi|^2\right)^{1-s}}\hfill\end{multline*}
	\vspace*{-.2cm}\begin{align*}
	&=N\int_{ \mathbb{R}^N}F(u)\dd x.
	\end{align*}
\end{theorem}

\noindent\textbf{A non-existence result.} Once the Pohozaev-type identity in $\mathbb{R}^N$ was obtained, we prove a result of non-existence of non-trivial solutions. Observe that our result is valid not only for positive solutions. 
\begin{theorem}\label{tnonexistence}The problem
	\[(-\Delta+m^2)^s u=|u|^{p-2}u\quad\text{in }\ \mathbb{R}^N\]
	has no non-trivial solution if $p\geq 2^{*}_s$, where
	\[2^{*}_s= \frac{2N}{N-2s}.\]
\end{theorem}
It follows from Theorem \ref{tnonexistence} that 
the constant
\[0<\Lambda=\inf_{u\in H^{s}(\mathbb{R}^N)\setminus\{0\}}\frac{\displaystyle\int_{ \mathbb{R}^N}(m^2+4\pi^2|\xi|^2)^{s}|\hat{u}(\xi)|^2\dd\xi}{\displaystyle\left(\int_{ \mathbb{R}^N}|u|^{2^{*}_s}\dd x\right)^{2/2^{*}_s}}<\infty\]
is not attained. 

However, Cotsiolis and Tavoularis \cite{Cotsiolis} proved that the Sobolev constant  
\[S=\inf_{u\in H^{s}(\mathbb{R}^N)\setminus\{0\}}\frac{\displaystyle\int_{ \mathbb{R}^N}(2\pi|\xi|)^{2s}\,|\hat{u}(\xi)|^2\dd\xi}{\displaystyle\left(\int_{ \mathbb{R}^N}|u|^{2^{*}_s}\dd x\right)^{2/2^{*}_s}}\]
is attained by the function $U(x)=c(\mu^2+(x-x_0)^2)^{-(N-2s)/2}$, where $c$ and $\mu$ are contants, with $c$ chosen so that  $|U|_{2^{*}_s}=1$. The function $U$ is in $ H^{s}(\mathbb{R}^N)$, if $N > 4s$. We prove that $\Lambda=S$.

\noindent\textbf{Solution to a asymptotic linear problem.} In the sequel, we handle the problem
\begin{equation}\label{problem}(-\Delta +m^2)^s u=f(u)\quad\text{in }\ \mathbb{R}^N,\end{equation}
with the non-linearity $f$ having as a model (see specific hypotheses in Section \ref{aplication})
\[f(t)=c\frac{t^3}{1+t^2},\]
where $c$ is a constant greater than $m^{2s}$. 

Considering the behavior of the non-linearity $f$ at infinity, it is natural to consider Cerami sequences and to apply the Ghoussoub-Preiss theorem. As a closed manifold, we consider a Nehari-Pohozaev manifold, as introduced by Ruiz \cite{Ruiz}, making use of the Pohozaev-type identity obtained for the space $\mathbb{R}^N$.
\begin{theorem}\label{texistence}The problem
\begin{equation}
(-\Delta +m^2)^s\, u=f(u)\quad\text{in }\ \mathbb{R}^N,
\end{equation}
when $f$ satisfies
\begin{enumerate}
\item [$(f_1)$] $f\colon\mathbb{R}\to \mathbb{R}$ is a $C^1$ function such that $f(t)/t$ is increasing if $t>0$ and decreasing if $t<0$;
\item [$(f_2)$] $\displaystyle\lim_{t\to 0} \frac{f(t)}{t}=0\quad\text{and}\quad \displaystyle\lim_{t\to \infty} \frac{f(t)}{t}=k\in (m^{2s},\infty]$;
\item [$(f_3)$] $\displaystyle\lim_{|t|\to\infty}tf(t)-2F(t)=\infty$, where $F(t)=\int_0^t f(\tau)\dd \tau$,
\end{enumerate}
has a ground state solution $w\in H^{s}(\mathbb{R}^N)$.
\end{theorem} 

\noindent\textbf{A symmetry result.} In Section \ref{Symmetry} we present some simple facts about a modified Bessel kernel. Then, applying the moving planes method in integral form as introduced by W. Chen, C. Li and B. Ou \cite{ChenLiOu}, we prove the following result:
\begin{theorem}\label{tsymmetry} Let $f\colon [0,\infty) \rightarrow \mathbb{R}$ be a continuous function that satisfies
\begin{enumerate}
\item[$(s_1)$] $f'(t) \geq 0 $ and $f''(t) \geq 0$ for all $t \in [0,\infty)$.
\item[$(s_2)$] For any $\beta\in (1, 2^{*}_{s} - 1)$, there exists $q \in [2,2^{*}_{s}]$ with $q > \max\{\beta,\frac{N(\beta-1)}{2s}\}$ such that  $f'(w) \in L^{q/(\beta-1)}(\mathbb{R}^{N}), \ \ \forall w \in H^{s}(\mathbb{R}^{N})$.
\end{enumerate}
For any $0<s<1$, $N>2s$ and $m\in \mathbb{R}\setminus\{0\}$, if $u(x)$ is a positive solution of
\begin{align*}
(-\Delta + m^2)^{s} u &= f(u) \ \ \mbox{in} \ \ \ \mathbb{R}^{N},
\end{align*}
then $u$ is radially symmetric and decreasing with respect to the origin.
\end{theorem}

Although hypothesis ($s_2$) is not standard, in that section we show that it is satisfied by functions like: a) $f(t)=t^\alpha$, if $\alpha\in (1,2^{*}_s-1)$;
b) $f(t)=t^{\alpha} + t^{\gamma}$, if $\alpha,\gamma\in (1,2^{*}_s-1)$;
c) $f(t)=t\ln(1+t)$.

\section{A Fourier transform theory for $(-\Delta+m^2)^s$}\label{prelim}
In the case $0<s<1$, aiming to present a specific Fourier theory for the fractionary operator $(-\Delta+m^2)^s$, we follow the work of Stinga and Torrea \cite{StingaTorrea},  our approach being based on the action of the heat semigroup  $e^{tL}$ generated by the operator $L$ acting on $L^s h$, for $h\in Dom(L^s)$. Some results about the Fourier transform can be found in many texts about the subject, see, e.g, \cite{Garofalo}. 

For any $\lambda>0$ and $0<s<1$, making the change of variables $s=\lambda t$ and integrating by parts, 
we conclude that
\begin{align}\label{intGamma}\lambda^s=\frac{1}{\Gamma(-s)}\int_0^\infty \left(e^{-\lambda t}-1\right)\frac{\dd t}{t^{1+s}}. \end{align}

Therefore, for an operator $L$, we have an expression for $L^s  f(x)$, if $f\colon\mathbb{R}^N\to\mathbb{R}$:
\begin{align}\label{Lsigma}
L^s  f(x)=\frac{1}{\Gamma(-s )}\int_0^\infty\left(e^{-Lt}f(x)-f(x)\right)\frac{\dd t}{t^{1+s }},\quad\forall\, x\in\mathbb{R}^N,\ \ \forall\, s \in (0,1).
\end{align}
This formula makes use of the classical heat-diffusion semigroup generated by $L$, see \cite[Equation (1.11)]{StingaTorrea}. 

Let us consider the problem
\begin{equation}\label{parabolic}\left\{\begin{array}{rlll}
v_t(x,t)&=&(\Delta-m^2)v(x,t), &(x,t)\in \mathbb{R}^N\times(0,\infty)\\
v(x,0)&=&f(x), &x\in\mathbb{R}^N,
\end{array}\right.\end{equation}
which has the solution
\[v(x,t)=e^{(\Delta-m^2)t}f(x),\quad\forall\,x\in\mathbb{R}^N,\ \ \forall\,t\in(0,\infty).\]

Taking the Fourier transform in \eqref{parabolic} with respect to the variable $x$, we obtain
\begin{equation*}\left\{\begin{array}{rlll}
\displaystyle\frac{\partial}{\partial t}\hat{v}(\xi, t)&=&-(m^2+4\pi^2|\xi|^2)\hat{v}(\xi,t), &(\xi,t)\in \mathbb{R}^N\times(0,\infty)\vspace*{.1cm}\\
\hat{v}(\xi,0)&=&\hat{f}(\xi), &\xi\in\mathbb{R}^N,
\end{array}\right.\end{equation*}
its solution being given by
\begin{equation*}
\hat{v}(\xi,t)=e^{-(m^2+4\pi^2|\xi|^2)t}\hat{f}(\xi),\quad(\xi,t)\in\mathbb{R}^N\times (0,\infty).
\end{equation*}

We conclude that
\begin{align}\label{ParabolicSol}
e^{(\Delta-m^2)t}f(x)&=v(x,t)=\mathcal{F}^{-1}\left(\hat{v}(\xi,t)\right)(x)
=\int_{ \mathbb{R}^N}e^{2\pi i x\cdot\xi}\hat{v}(\xi,t)\dd\xi\nonumber\\
&=\int_{ \mathbb{R}^N}e^{2\pi i x\cdot\xi}e^{-(m^2+4\pi^2|\xi|^2)t}\hat{f}(\xi)\dd\xi.
\end{align}

On the other hand, substituting $L=(-\Delta+m^2)$ into \eqref{Lsigma}, we obtain
\[(-\Delta+m^2)^s  f(x)=\frac{1}{\Gamma(-s )}\int_0^\infty\left(e^{(\Delta-m^2)t}f(x)-f(x)\right)\frac{\dd t}{t^{1+s }}.\]

Since $\mathcal{F}^{-1}(\mathcal{F}f)=f$, it follows from \eqref{ParabolicSol} that
\begin{multline*}
(-\Delta+m^2)^s  f(x)\\
\end{multline*}\vspace*{-1cm}
\begin{align*}
&=\frac{1}{\Gamma(-s )}\int_0^\infty\left[\int_{ \mathbb{R}^N}e^{2\pi i x\cdot\xi}e^{-(m^2+4\pi^2|\xi|^2)t}\hat{f}(\xi)\dd\xi-\int_{ \mathbb{R}^N}e^{2\pi ix\cdot\xi}\hat{f}(\xi)\dd\xi\right]\frac{\dd t}{t^{1+s }}\\
&=\frac{1}{\Gamma(-s )}\int_{ \mathbb{R}^N}e^{2\pi i x\cdot\xi}\hat{f}(\xi)\left[\int_0^\infty \left(e^{-(m^2+4\pi^2|\xi|^2)t}-1\right)\frac{\dd t}{t^{1+s }}\right]\dd\xi.
\end{align*}

Since \eqref{intGamma} yields
\[\int_0^\infty \left(e^{-(m^2+4\pi^2|\xi|^2)t}-1\right)\frac{\dd t}{t^{1+s }}=\Gamma(-s )\,(m^2+4\pi^2|\xi|^2)^{s },\]
we obtain
\begin{align}\label{-Delta+m2sigma}
(-\Delta+m^2)^s  f(x)&=\int_{ \mathbb{R}^N}e^{2\pi i x\cdot\xi}\,(m^2+4\pi^2|\xi|^2)^{s }\hat{f}(\xi)\dd\xi\nonumber\\
&=\mathcal{F}^{-1}\left((m^2+4\pi^2|\cdot|^2)^s  \hat{f}\right)(x).
\end{align}

Therefore, we conclude that
\begin{align}\label{-Delta+m2sigmaF}\mathcal{F}\left[(-\Delta+m^2)^s  f\right](\xi)=(m^2+4\pi^2|\xi|^2)^s \mathcal{F}(f(\xi)).\end{align}

In the case $s=1/2$, formula \eqref{-Delta+m2sigmaF} can be found in \cite{LiebLoss}. In the general case, it is no surprise, see Stein \cite{Stein} or Garofalo \cite{Garofalo}. The same happens with the next result:
\begin{lemma}\label{lemma1f}For any $f\in\mathcal{S}(\mathbb{R}^N)$ and $s \in (0,1)$ we have
\[(-\Delta+m^2)^s  f(\xi)=\mathcal{F}\left[(m^2+4\pi^2|\xi|^2)^s \mathcal{F}^{-1}(f)\right](\xi).\]
\end{lemma}
\begin{proof}Since $g=\mathcal{F}^{-1}(\mathcal{F}(g))$ for any $g\in\mathcal{S}(\mathbb{R}^N)$, it follows from \eqref{-Delta+m2sigma} that
\begin{align*}
(-\Delta+m^2)^s  f(\xi)
&=\int_{\mathbb{R}^N}e^{2\pi ix\cdot\xi}\,(m^2+4\pi^2|x|^2)^s \hat{f}(x)\dd x\\
&=\int_{\mathbb{R}^N}\int_{ \mathbb{R}^N}e^{2\pi ix\cdot\xi}\,(m^2+4\pi^2|x|^2)^s  e^{-2\pi ix\cdot z}f(z)\dd z\dd x\\
\end{align*}\vspace*{-.8cm}
	
The change of variables $x\to -x$ yields
\begin{align*}
(-\Delta+m^2)^s  f(\xi)&=\int_{ \mathbb{R}^N}e^{-2\pi ix\cdot\xi}\,(m^2+4\pi^2|x|^2)^s \left(\int_{ \mathbb{R}^N}e^{2\pi ix\cdot z}f(z)\dd z\right)\dd x\\
&=\mathcal{F}\left[(m^2+4\pi^2|\cdot|^2)^s \mathcal{F}^{-1}(f)\right](\xi)
\end{align*}
and the proof is complete.
$\hfill\Box$\end{proof}

We recall that, for any $f,g\in L^1(\mathbb{R}^N)$ it holds
\begin{align}\label{fhatg=hatfg}\int_{ \mathbb{R}^N}\hat{f}(x)g(x)\dd x=\int_{ \mathbb{R}^N}f(x)\hat{g}(x)\dd x.
\end{align}

We now prove the symmetry of the operator $(-\Delta+m^2)^s$:
\begin{lemma}\label{symmetry} For any $s\in(0,1)$ and $u,v\in\mathcal{S}(\mathbb{R}^N)$ we have
\[\int_{\mathbb{R}^N}\left[(-\Delta+m^2)^s  u(x)\right] v(x)\dd x=\int_{ \mathbb{R}^N}u(x)\left[(-\Delta+m^2)^s  v(x)\right]\dd x.\]
\end{lemma}
\begin{proof}We have 
\begin{align*}
\int_{\mathbb{R}^N}\left[(-\Delta +m^2)^s  u(x)\right]v(x)\dd x&=\int_{ \mathbb{R}^N}(-\Delta +m^2)^s  u(x)\mathcal{F}\left(\mathcal{F}^{-1}(v)\right)(x)\dd x\\
&=\int_{ \mathbb{R}^N}\mathcal{F}\left((-\Delta +m^2)^s  u\right)(\xi)\left(\mathcal{F}^{-1}(v)\right)(\xi)\dd\xi\\
&=\int_{ \mathbb{R}^N}(m^2+4\pi^2|\xi|^2)^{s }\mathcal{F}(u)\left(\mathcal{F}^{-1}(v)\right)(\xi)\dd\xi,
	\end{align*}
where the second equality follows from \eqref{fhatg=hatfg} and the third by \eqref{-Delta+m2sigmaF}. Since Lemma \ref{lemma1f} guarantees that $(m^2+4\pi^2|\xi|^2)^s \mathcal{F}^{-1}(v)(\xi)=\mathcal{F}^{-1}\left[(-\Delta+m^2)^s  v\right](\xi)$,
we obtain
\begin{align*}\int_{\mathbb{R}^N}\left[(-\Delta +m^2)^s  u(x)\right]v(x)\dd x&=\int_{\mathbb{R}^N}\mathcal{F}(u)(\xi)\mathcal{F}^{-1}\left[(-\Delta+m^2)^s  v\right](\xi)\\
&=\int_{ \mathbb{R}^N}u(x)\left[(-\Delta+m^2)^s  v(x)\right]\dd x,\end{align*}
the last equality being a consequence of \eqref{fhatg=hatfg}.
$\hfill\Box$\end{proof}
\begin{lemma}\label{parts}For any $s_1,s_2\in (0,1)$ such that $s_1+s_2<1$ it holds
\[(-\Delta+m^2)^{s_1}\cdot (-\Delta+m^2)^{s_2}=(-\Delta+m^2)^{s_1+s_2}.\]
\end{lemma}
\begin{proof}For any $u\in \mathcal{S}(\mathbb{R}^N)$ we have
\begin{align*}
\mathcal{F}\left((-\Delta+m^2)^{s_1+s_2}u\right)(\xi)&=(m^2+4\pi^2|\xi|^2)^{s_1+s_2}\mathcal{F}(u)(\xi)\\
&=(m^2+4\pi^2|\xi|^2)^{s_1}(m^2+4\pi^2|\xi|^2)^{s_2}\mathcal{F}(u)(\xi)\\
&=(m^2+4\pi^2|\xi|^2)^{s_1}\mathcal{F}\left((-\Delta+m^2)^{s_2}u\right)(\xi)\\
&=\mathcal{F}\left((-\Delta+m^2)^{s_1}\left((-\Delta+m^2)^{s_2}u\right)\right)(\xi).
\end{align*}
Taking $\mathcal{F}^{-1}$, we conclude.
$\hfill\Box$\end{proof}

By applying Lemmas \ref{parts} and \ref{symmetry} we immediately obtain:
\begin{corollary}\label{corsolution}
For any $s\in (0,1)$ and $u,v\in\mathcal{S}(\mathbb{R}^N)$ it holds
\[\int_{\mathbb{R}^N}\left[(-\Delta+m^2)^su(x)\right]v(x)\dd x=\int_{\mathbb{R}^N}(-\Delta+m^2)^{s/2}u(x)(-\Delta+m^2)^{s/2}v(x)\dd x.\]
\end{corollary}
\section{A Pohozaev-type identity for the extension problem}\label{Prooft1}

We consider the extension problem in $ \mathbb{R}^{N+1}_+=\mathbb{R}^N\times (0,\infty)$
\begin{equation}\label{Extension}
\left\{\begin{array}{rcll}
-\diver(y^{1-2s}\nabla w)+m^2y^{1-2s}w&=&0, &(x,y)\in \mathbb{R}^{N+1}_+,\\
\displaystyle\lim_{y\to 0^+}\left(-y^{1-2s}\frac{\partial w}{\partial y}(x,y)\right)&=&{k_s}f(w(x,0)), &x\in \mathbb{R}^N.
\end{array}\right.
\end{equation}

Usually, changing scales, the constant $k_s$ is assumed to be equal to $1$. However, in order to better understand the behavior while obtaining the Pohozaev-type identity, we will not change scales. 

For all $R>0$ and $\delta\in (0,R)$, define
\begin{align*}
D^+_{R,\delta}&=\left\{z=(x,y)\in\mathbb{R}^N\times [\delta,\infty)\,:\,|z|^2\leq R^2\right\}\\
\partial D^1_{R,\delta}&=\left\{z=(x,y)\in\mathbb{R}^N\times\{y=\delta\}\,:\,|x|^2\leq R^2-\delta^2\right\}\\
\partial D^2_{R,\delta}&=\left\{z=(x,y)\in\mathbb{R}^N\times[\delta,\infty)\,:\,|z|^2= R^2\right\}
\end{align*}
and note that $\partial D^+_{R,\delta}=\partial D^1_{R,\delta}\cup \partial D^2_{R,\delta}$. 

Denoting $\eta$ the unit outward normal vector to $\partial D^+_{R,\delta}$, then
\[\eta=\left\{\begin{array}{ll}
(0,\ldots,0,-1), &\text{in }\ \partial D^1_{R,\delta}\vspace*{.1cm}\\
\displaystyle\frac{z}{R},&\text{in }\ \partial D^2_{R,\delta}\end{array}\right.\]

Since 
\begin{multline*}
\diver [y^{1-2s}\nabla w](z\cdot \nabla w)\hfill
\end{multline*}
\begin{align*}&=\diver [y^{1-2s}\nabla w\,(z\cdot \nabla w)]-\left[y^{1-2s}\nabla w\cdot\nabla (z\cdot \nabla w)\right]\\
&=\diver\!\left[y^{1-2s}\nabla w\,(z\cdot\nabla w)-y^{1-2s}z\frac{|\nabla w|^2}{2}\right]+\frac{N-2s}{2}y^{1-2s}|\nabla w|^2,
\end{align*}
multiplication of \eqref{Extension} by $w\cdot\nabla w$ and integration on $D^+_{R,\delta}$ give
\begin{align*}
0
&=-\iint_{D^+_{R,\delta}}\diver\left[y^{1-2s}\nabla w\,(z\cdot\nabla w)-y^{1-2s}z\frac{|\nabla w|^2}{2}\right]\dd x\dd y\\
&\qquad-\frac{N-2s}{2}\iint_{D^+_{R,\delta}}y^{1-2s}|\nabla w|^2\dd x\dd y+\iint_{D^+_{R,\delta}}m^2y^{1-2s}w\,(z\cdot\nabla w)\dd x\dd y
\end{align*}
so that the application of the divergence theorem yields
\begin{align*}
0&=-\int_{\partial D^+_{R,\delta}}\left[y^{1-2s}(\nabla w\cdot \eta)\,(z\cdot\nabla w)-y^{1-2s}(z\cdot \eta)\frac{|\nabla w|^2}{2}\right]\dd \sigma\\
&\qquad -\frac{N-2s}{2}\iint_{D^+_{R,\delta}}y^{1-2s}|\nabla w|^2\dd x\dd y+\iint_{D^+_{R,\delta}}m^2y^{1-2s}w\,(z\cdot\nabla w)\dd x\dd y\\
&=J_1(R,\delta)+J_2(R,\delta)+J_3(R,\delta)+J_4(R,\delta),
\end{align*}
where
\begin{align*}
J_1(R,\delta)&=-\int_{\partial D^1_{R,\delta}}y^{1-2s}\left[(-\partial_y w)\,(z\cdot\nabla w)+y\frac{|\nabla w|^2}{2}\right]\dd\sigma\\
J_2(R,\delta)&=-\int_{\partial D^2_{R,\delta}}y^{1-2s}\left[\frac{1}{R}(z\cdot\nabla w)^2-R\frac{|\nabla w|^2}{2}\right]\dd\sigma\\
J_3(R,\delta)&=-\frac{N-2s}{2}\iint_{D^+_{R,\delta}}y^{1-2s}|\nabla w|^2\dd x\dd y\\
J_4(R,\delta)&=\iint_{D^+_{R,\delta}}m^2y^{1-2s}w\,(z\cdot\nabla w)\dd x\dd y.
\end{align*}

Since $u(x)=w(x,0)$, then $f(u)(x\cdot \nabla u)=\diver [F(u)x]-NF(u)$.

Denote $B_R=\{(x,0)\in\mathbb{R}^{N+1}_+\,:\,|x|^2\leq R^2\}$. Then, \eqref{Extension} and the divergence theorem imply that
\begin{multline*}
\lim_{\delta\to 0}\int_{\partial D^1_{R,\delta}}y^{1-2s}(-\partial_y w)\,(z\cdot\nabla w)d\sigma\hfill
\end{multline*}
\begin{align*}
&=\int_{B_R}k_sf(u)(x\cdot \nabla u)\dd x=k_s\int_{B_R}\left[\diver[xF(u)]-NF(u)\right]\dd x\\
&=k_s\int_{\partial B_R}F(u)(x\cdot \eta)\dd\sigma-k_sN\int_{B_R}F(u)\dd x.
\end{align*}
It follows that
\begin{equation}\label{J1}\lim_{\delta\to 0}J_1(R,\delta)=-\left[k_s\int_{\partial B_R}F(u)(x\cdot \eta)\dd\sigma-k_sN\int_{B_R}F(u)\dd x\right].\end{equation}

Now we observe that
\[\left|\int_{\partial B_R}F(u)(x\cdot \eta)\dd\sigma\right|\leq R \int_{\partial B_R}|F(u)|\dd\sigma\]
and
\[\left|\int_{\partial D^2_{R,\delta}}y^{1-2s}\left[\frac{1}{R}(z\cdot\nabla w)^2-R\frac{|\nabla w|^2}{2}\right]\dd\sigma\right|\leq2 R\int_{\partial D^2_{R,\delta}}y^{1-2s}|\nabla w|^2\dd\sigma.\]

We claim that there exists a sequence $(R_n)$ such that, when $R_n\to \infty$,
\begin{equation}\label{Rn}\lim_{n\to\infty}R_n\int_{\partial B_{R_n}}|F(u)|\dd\sigma=0=\lim_{n\to\infty}R_n\int_{\partial D^2_{R,\delta}}y^{1-2s}|\nabla w|^2\dd\sigma.\end{equation}

Supposing the contrary, there exist $\tau>0$ 
%
$R_1>0$ such that, for all $R\geq R_1$,
\[\int_{\partial B_{R_1}}|F(u)|\dd\sigma\geq \frac{\tau}{R},\]
from what follows
\[\int_{\mathbb{R}^N}|F(u)|\dd x\geq \int_{R_1}^\infty\int_{\partial B_{R_1}}|F(u)|\dd\sigma\dd R\geq \int_{R_1}^\infty\frac{\tau}{R}\dd R=\infty,\]
a contradiction. The same argument also applies to the second integral in \eqref{Rn} and proves the claim, which yields not only that 
\[\lim_{n\to\infty}\lim_{\delta\to 0}J_1(R,\delta)=-k_sN\int_{\mathbb{R}^N}F(u)\dd x\]
but also
\[\lim_{n\to\infty}\lim_{\delta\to 0}J_2(R,\delta)=0.\]

By considering the same sequence $(R_n)$, it holds
\[\lim_{n\to\infty}\lim_{\delta\to 0}J_3(R_n,\delta)=-\frac{N-2s}{2}\iint_{\mathbb{R}^{N+1}_+}y^{1-2s}|\nabla w|^2\dd x\dd y.\]

We now analyze $J_4(R,\delta)$:
\begin{align*}J_4(R,\delta)&=\iint_{D^+_{R,\delta}}m^2y^{1-2s}w\,(z\cdot\nabla w)\dd x\dd y\\
&=m^2\iint_{D^+_{R,\delta}}y^{1-2s}\left(z\cdot\nabla \left(\frac{w^2}{2}\right)\right)\dd x\dd y.\end{align*}

For this, we consider the field $\Phi=y^{1-2s}\frac{w^2}{2}z$. Since
\begin{align*}
\diver \Phi
&=\frac{N+2-2s}{2}y^{1-2s}w^2+y^{1-2s}\left(z\cdot \nabla\left(\frac{w^2}{2}\right)\right),
\end{align*}
the divergence theorem yields
\begin{multline*}
m^2\iint_{D^+_{R,\delta}}y^{1-2s}\left(z\cdot\nabla \left(\frac{w^2}{2}\right)\right)\dd x\dd y\hfill
\end{multline*}
\begin{align*}
&=-m^2\frac{N+2-2s}{2}\iint_{D^+_{R,\delta}}y^{1-2s}w^2\dd x\dd y-m^2\int_{\partial D^1_{R,\delta}}y^{1-2s}y\frac{w^2}{2}\dd\sigma\\
&\qquad+m^2\int_{\partial D^2_{R,\delta}}y^{1-2s}R\frac{w^2}{2}\dd\sigma.
\end{align*}

The same argument applied before shows that
\[\lim_{n\to\infty}\lim_{\delta\to 0}J_4(R,\delta)=-m^2\frac{N+2-2s}{2}\iint_{\mathbb{R}^{N+1}_+}y^{1-2s}w^2\dd x\dd y.\]

Collecting our results, we conclude the Pohozaev-type identity in $\mathbb{R}^{N+1}_+$
\begin{multline*}
\frac{N-2s}{2}\iint_{ \mathbb{R}^{N+1}_+}y^{1-2s}|\nabla w|^2\dd x\dd y+m^2\frac{N+2-2s}{2}\iint_{\mathbb{R}^{N+1}_+}y^{1-2s}w^2\dd x\dd y\hfill\end{multline*}
\vspace*{-.3cm}\begin{align}\label{Poho1}
&=Nk_s\int_{ \mathbb{R}^{N}}F(w(x,0))\dd x.
\end{align}

Of course, equation \eqref{Poho1} is direct applicable to problem \eqref{P}. As mentioned before, it is not difficult to consider $f(x,u)$ instead of $f(u)$. 

\section{A Pohozaev-type identity for the problem in $\mathbb{R}^N$}
In this section we obtain a Pohozaev-type identity for the problem
\[(-\Delta+m^2)^su=f(u)\ \ \text{in }\ \mathbb{R}^N.\]

Since problem \eqref{P} is an extension of problem \eqref{original}, it satisfies
\begin{equation}\label{inital}\left\{\begin{array}{ll}
\displaystyle\Delta_x w+\frac{1-2s}{y}w_y+w_{yy}-m^2w=0 &\text{in }\ \mathbb{R}^{N+1}_+\vspace{.1cm} \\
w(x,0)=u(x), &x\in \mathbb{R}^{N},\end{array}\right.
\end{equation}
In particular, $F(w(x,0))=F(u)$ and the right-hand side of \eqref{Poho1} causes no problem.

We now interpret the integrals in $ \mathbb{R}^{N+1}_+$ as integrals in $\mathbb{R}^{N}$. We start with a technical result, which has a tricky proof.
\begin{lemma}\label{l1}If $\mathcal{K}(\Phi)$ is the functional given by \eqref{KPhi}, then
\[\mathcal{K}(\Phi_s)=k_s.\]
\end{lemma}
\begin{proof}Since $\Phi_s$ satisfies \eqref{Phi}, we have
\[\Phi^2_s=\frac{1-2s}{y}\Phi'_s\Phi_s+\Phi''_s\Phi_s,\]
from what follows
\begin{align*}
\Phi^2_sy^{1-2s}&=(\Phi'_s\Phi_s)\frac{\dd}{\dd y}y^{1-2s}+(\Phi''_s\Phi_s) y^{1-2s}
=\frac{\dd}{\dd y}(\Phi'_s\Phi_s y^{1-2s})-(\Phi'_s)^2y^{1-2s},
\end{align*}
thus showing that
\[[\Phi^2_s+(\Phi'_s)^2]y^{1-2s}=\frac{\dd}{\dd y}(\Phi_s\Phi'_sy^{1-2s}).\]
	
	Therefore,
\begin{align*}
\int_0^\infty [\Phi^2_s+(\Phi'_s)^2]y^{1-2s}\dd y=\lim_{y\to 0}-\Phi_s\Phi'_sy^{1-2s}=-\lim_{y\to 0}\Phi'_sy^{1-2s}=2sc_1=k_s
\end{align*}
and we are done.
$\hfill\Box$\end{proof}

A second technical result that will be necessary in our analysis is the following:
\begin{lemma}\label{l2}It holds
\[\int_0^\infty [\Phi_s(t)]^2t^{1-2s}\dd t=sk_s.\]
\end{lemma}
\begin{proof}Integration by parts yields
\begin{multline*}
\int_0^\infty [\Phi_s(t)]^2t^{1-2s}\dd t\hfill
\end{multline*}
\begin{align*}
&=-\frac{1}{1-s}\int_0^\infty\Phi'_s(t)\left(\frac{1-2s}{t}\Phi'_s(t)+\Phi''_s(t)\right)t^{2-2s}\dd t\\
&=-\frac{1-2s}{1-s}\int_0^\infty [\Phi'_s(t)]^2t^{1-2s}\dd t-\frac{1}{1-s}\int_0^\infty \Phi'_s(t)\Phi''_s(t)t^{2-2s}\dd t\\
&=-\frac{1-2s}{1-s}\int_0^\infty [\Phi'_s(t)]^2t^{1-2s}\dd t+\int_0^\infty [\Phi'_s(t)]^2t^{1-2s}\dd t\\
&=\frac{s}{1-s}\left[\int_0^\infty[\Phi'_s(t)]^2t^{1-2s}\dd t+\int_0^\infty[\Phi_s(t)]^2t^{1-2s}\dd t\right]-\frac{s}{1-s}\int_0^\infty[\Phi_s(t)]^2t^{1-2s}\dd t\\
&=\frac{s}{1-s}k_s-\frac{s}{1-s}\int_0^\infty[\Phi_s(t)]^2t^{1-2s}\dd t,
\end{align*}
from what follows our result.
$\hfill\Box$\end{proof}

We are now in position to translate the Pohozaev-type identity in $\mathbb{R}^{N+1}_+$ into terms of integrals in $u$. We begin by writing the left-hand side of \eqref{Poho1} as
\[\frac{N-2s}{2}\iint_{ \mathbb{R}^{N+1}_+}y^{1-2s}\left[|\nabla w|^2+m^2w^2\right]\dd x\dd y+m^2\iint_{\mathbb{R}^{N+1}_+}y^{1-2s}w^2\dd x\dd y.\]

Observe that
\begin{align*}
\int_{\mathbb{R}^{N}}|\nabla w(x,y)|^2\dd x&=\int_{\mathbb{R}^{N}}\left(|\nabla_x w(x,y)|^2+\left|\frac{\partial w}{\partial y}(x,y)\right|^2\right)\dd x\\
&=\int_{\mathbb{R}^{N}}\left(4\pi^2|\xi|^2|\hat{w}(\xi,y)|^2+\left|\frac{\partial \hat{w}}{\partial y}(\xi,y)\right|^2\right)\dd \xi.
\end{align*}

By making use of the expression for $\hat{w}(\xi,y)$ given by \eqref{solFourier} and denoting $c=\sqrt{m^2+4\pi^2|\xi|^2}$, by multiplying the last equality by $y^{1-2s}$ and integrating in $y$ we obtain
\begin{multline*}
\iint_{\mathbb{R}^{N+1}_+}\left(|\nabla w(x,y)|^2+m^2|w(x,y)|^2\right)y^{1-2s}\dd x\dd y
\end{multline*}\vspace*{-.3cm}
\begin{align*}
&=\int_0^{\infty}\int_{\mathbb{R}^{N}}|\nabla w(x,y)|^2y^{1-2s}\dd x\dd y+\int_0^\infty\int_{\mathbb{R}^{N}}m^2|w(x,y)|^2y^{1-2s}\dd x\dd y\\
&=\int_0^\infty\int_{\mathbb{R}^{N}}\left(c^2|\hat{w}(\xi,y)|^2+\left|\frac{\partial \hat{w}}{\partial y}(\xi,y)\right|^2\right)y^{1-2s}\dd \xi\dd y\\
&=\int_0^\infty\int_{\mathbb{R}^{N}}\left(c^2|\hat{u}(\xi)|^2|\Phi_s(cy)|^2+|\hat{u}(\xi)c\,\Phi'_s(cy)|^2\right)y^{1-2s}\dd \xi\dd y\\
&=\int_0^\infty\int_{\mathbb{R}^{N}}c^2|\hat{u}(\xi)|^2\left(|\Phi_s(cy)|^2+|\Phi'_s(cy)|^2\right)y^{1-2s}\dd \xi\dd y,\\
&=\int_{\mathbb{R}^{N}}c^{2s}|\hat{u}(\xi)|^2\dd\xi\left(\int_0^{\infty}\left(|\Phi_s(t)|^2+|\Phi'_s(t)|^2\right)t^{1-2s}\dd t\right)\\
&=\mathcal{K}(\Phi_s)\int_{\mathbb{R}^{N}}c^{2s}|\hat{u}(\xi)|^2\dd\xi
=k_s\int_{\mathbb{R}^{N}}\left(4\pi^2|\xi|^2+m^2\right)^{s}|\hat{u}(\xi)|^2\dd\xi,
\end{align*}
as a consequence of Lemma \ref{l1}.

We conclude that
\begin{equation}\label{intnorm}\iint_{ \mathbb{R}^{N+1}_+}y^{1-2s}\left[|\nabla w|^2+m^2w^2\right]\dd x\dd y=k_s\int_{ \mathbb{R}^N}\left|(m^2-\Delta)^{s/2}u(x)\right|^2\dd x.\end{equation}

By applying  Plancherel's identity, we interpret the last integral in $w$ in \eqref{Poho1} as a integral in $\mathbb{R}^{N}$.
\begin{multline*}
m^2\!\iint_{\mathbb{R}^{N+1}_+}y^{1-2s}w^2\dd x\dd y\hfill
\end{multline*}
\begin{align*}
&=m^2\!\int_{ \mathbb{R}^N}|\hat{u}(\xi)|^2\dd\xi\int_0^\infty\Phi^2_s\left(\sqrt{m^2+4\pi^2|\xi|^2}\,y\right)y^{1-2s}\dd y\dd\xi.\end{align*}

Changing variables, we obtain that
\begin{multline*}
\int_0^\infty\Phi^2_s\left(\sqrt{m^2+4\pi^2|\xi|^2}\,y\right)y^{1-2s}\dd y\hfill
\end{multline*}
\vspace*{-.3cm}\begin{align*}&=\int_{0}^{\infty}[\Phi_s(t)]^2\frac{t^{1-2s}}{\left(m^2+4\pi^2|\xi|^2\right)^{(1-2s)/2}}\frac{\dd t}{\left(m^2+4\pi^2|\xi|^2\right)^{1/2}},
\end{align*}
so that
\begin{align*}m^2\iint_{\mathbb{R}^{N+1}_+}y^{1-2s}w^2\dd x\dd y&=m^2\int_{ \mathbb{R}^N}\frac{|\hat{u}(\xi)|^2}{\left(m^2+4\pi^2|\xi|^2\right)^{1-s}}\dd\xi\int_0^\infty [\Phi_s(t)]^2t^{1-2s}\dd t.
\end{align*}

It follows from Lemma \ref{l2} the desired Pohozaev-type identity:
\begin{multline*}\frac{N-2s}{2}k_s\int_{ \mathbb{R}^N}\left|(m^2-\Delta)^{s/2}u(x)\right|^2\dd x+sk_sm^2\int_{ \mathbb{R}^N}\frac{|\hat{u}(\xi)|^2}{\left(m^2+4\pi^2|\xi|^2\right)^{1-s}}\dd\xi\hfill\end{multline*}
\vspace*{-.2cm}\begin{align}\label{Poho}
&=Nk_s\int_{ \mathbb{R}^N}F(u)\dd x,
\end{align}
thus showing that the Pohozaev-type identity does not depend on $k_s$.

\section{A non-existence result}\label{Prooft2}

In this section we show that the problem
\begin{equation}\label{inexistence}
(-\Delta +m^2)^s\, u=|u|^{p-2}u\quad\text{in }\ \mathbb{R}^N\end{equation}
has no solution $u\neq 0$ if $p\geq 2^{*}_s$.

Applying the Pohozaev-type identity \eqref{Poho} to the problem \eqref{inexistence}, we obtain
\begin{multline*}\frac{N-2s}{2}\int_{\mathbb{R}^{N}}|(-\Delta+m^2)^{s/2}u|^2\dd x+m^2 s\int_{\mathbb{R}^N}\frac{|\hat{u}(\xi)|^2}{(m^2+4\pi^2|\xi|^2)^{1-s}} \dd\xi\hfill
\end{multline*}
\begin{align}\label{inexistence1}
&=\frac{N}{p}\int_{\mathbb{R}^N}|u|^p\dd x.
\end{align}

Since $u$ is a solution of \eqref{inexistence}, it satisfies Corollary \ref{corsolution} for $f(u)=|u|^{p-2}u$. Thus,
\[\int_{ \mathbb{R}^N}(-\Delta+m^2)^{s/2}u(-\Delta+m^2)^{s/2}v \dd x=\int_{ \mathbb{R}^N}|u|^{p-2}uv\,\dd x\]
for any $v\in H^{s}(\mathbb{R}^N)$. Choosing $v=u$, we have
\begin{equation*}\label{inexistence2}
\int_{ \mathbb{R}^N}|(-\Delta+m^2)^{s/2}u|^2\dd x=\int_{ \mathbb{R}^N}|u|^p\dd x.
\end{equation*}
Substituting \eqref{inexistence2} into \eqref{inexistence1}, we obtain
\[\left(\frac{N}{p}-\frac{N-2s}{2}\right)\int_{ \mathbb{R}^N}|u|^p\dd x=m^2s\int_{\mathbb{R}^N}\frac{|\hat{u}(\xi)|^2}{(m^2+4\pi^2|\xi|^2)^{1-s}}\dd\xi>0,\]
from what follows that $\left(\frac{N}{p}-\frac{N-2s}{2}\right)>0$ and thus $p<2N/(N-2s)=2^{*}_s$.

As a consequence, the constant
\[0<\Lambda=\inf_{u\in H^{s}(\mathbb{R}^N)\setminus\{0\}}\frac{\displaystyle\int_{ \mathbb{R}^N}(m^2+4\pi^2|\xi|^2)^{s}|\hat{u}(\xi)|^2\dd\xi}{\displaystyle\left(\int_{ \mathbb{R}^N}|u|^{2^{*}_s}\dd x\right)^{2/2^{*}_s}}<\infty\]
is not attained.

It is well-known (see Cotsiolis and Tavoularis \cite{Cotsiolis}) that the function $U(x)=c(\mu^2+(x-x_0)^2)^{-(N-2s)/2}$, where $c$ and $\mu$ are constants, with $c$ chosen so that  $|U|_{2^{*}_s}=1$, attains the Sobolev constant
\[S=\inf_{u\in H^{s}(\mathbb{R}^N)\setminus\{0\}}\frac{\displaystyle\int_{ \mathbb{R}^N}(2\pi|\xi|)^{2s}\,|\hat{u}(\xi)|^2\dd\xi}{\displaystyle\left(\int_{ \mathbb{R}^N}|u|^{2^{*}_s}\dd x\right)^{2/2^{*}_s}}=\int_{ \mathbb{R}^N}(2\pi|\xi|)^{2s}\,|\hat{U}(\xi)|^2\dd\xi.\]
If $N > 4s$, then $U\in H^{s}(\mathbb{R}^N)$. 

We will show that $\Lambda=S$.

Of course, we have $\Lambda\geq S$, since
\[\int_{ \mathbb{R}^N}(m^2+4\pi^2|\xi|^2)^{s}|\hat{u}(\xi)|^2\dd\xi\geq \int_{ \mathbb{R}^N}(2\pi|\xi|)^{2s}\,|\hat{u}(\xi)|^2\dd\xi,\quad\forall\, u\in H^{s}(\mathbb{R}^N).\]

In order to show the opposite inequality, we define $v_t(x)=U(tx)$ for $t>0$. Changing variables, we obtain:\\
\begin{align*}(i)& \int_{\mathbb{R}^N}|\hat{v}_t(\xi)|^2\dd\xi=t^{-2N}\int_{\mathbb{R}^N}|\hat{U}(\xi/t)|^2\dd\xi=t^{-N}\int_{\mathbb{R}^N}|\hat{U}(\xi)|^2\dd\xi;\\
(ii)& \int_{\mathbb{R}^N}(2\pi|\xi|)^{2s}|\hat{v}_t(\xi)|^2\dd\xi=t^{-2N}\int_{\mathbb{R}^N}(2\pi|\xi|)^{2s}|\hat{U}(\xi/t)|^2\dd\xi\\
&\qquad\qquad=t^{-N+2s}\int_{\mathbb{R}^N}(2\pi)^{2s}|\xi||\hat{U}(\xi)|^2\dd\xi\\
(iii)&\int_{\mathbb{R}^N}|v_t(x)|^{2^{*}_s}\dd x=t^{-N}\int_{\mathbb{R}^N}|U(x)|^{2^{*}_s}\dd x=t^{-N}.
\end{align*}

It immediately follows from ($iii$) that
\[\left(\int_{\mathbb{R}^N}|v_t(x)|^{2^{*}_s}\dd x\right)^{2/2^{*}_s}=t^{-N+2s}\left(\int_{\mathbb{R}^N}|U(x)|^{2^{*}_s}\dd y\right)^{2/2^{*}_s}=t^{-N+2s}.\]

Thus,
\begin{align*}
\Lambda&\leq\frac{\displaystyle\int_{\mathbb{R}^N}(m^2+4\pi^2|\xi|^2)^{s}|\hat{v}_t(\xi)|^2\dd\xi}{\left(\displaystyle\int_{\mathbb{R}^N}|v_t|^{2^{*}_s}\dd x\right)^{2/2^{*}_s}}\\
&\leq\frac{m^{2s}\displaystyle\int_{\mathbb{R}^N}|\hat{v}_t(\xi)|^2\dd\xi+\int_{\mathbb{R}^N}(2\pi|\xi|)^{2s}\,|\hat{v}_t(\xi)|^2\dd\xi}{t^{-N+2s}}\\
&=\frac{m^{2s}t^{-N}|U|^2_2}{t^{-N+2s}}+\frac{t^{-N+2s}\displaystyle\int_{\mathbb{R}^N}(2\pi|\xi|)^{2s}\,|\hat{v}_t(\xi)|^2\dd\xi}{t^{-N+2s}}\\
&=\frac{m^{2s}}{t^{2s}}|U|^2_2+S.
\end{align*}
Making $t\to\infty$, we obtain $\Lambda\leq S$, completing the proof of $\Lambda= S$.

\section{The solution of a asymptotic linear problem}\label{aplication}
In this section we will prove existence of solution for the problem
\begin{equation}\label{prn}
(-\Delta +m^2)^s\, u=f(u)\quad\text{in }\ \mathbb{R}^N.
\end{equation}
when $f$ satisfies
\begin{enumerate}
\item [$(f_1)$] $f\colon\mathbb{R}\to \mathbb{R}$ is a $C^1$ function such that $f(t)/t$ is increasing if $t>0$ and decreasing if $t<0$;
\item [$(f_2)$] $\displaystyle\lim_{t\to 0} \frac{f(t)}{t}=0\quad\text{and}\quad\lim_{t\to \infty}\frac{f(t)}{t}=k\in (m^{2s},\infty]$;
\item [$(f_3)$] $\displaystyle\lim_{|t|\to\infty}tf(t)-2F(t)=\infty$, where $F(t)=\int_0^t f(\tau)\dd \tau$.
\end{enumerate}

Our hypotheses on $f$ imply that the non-quadraticity condition is satisfied by our problem, that is, $tf(t)-2F(t)>0$ for all $t\neq 0$ and ($f_3$). A model problem is given by
\[f(t)=c\frac{t^3}{1+t^2},\]
where $c>m^{2s}$ is a constant.

It follows from our hypotheses that
\begin{equation}\label{fandF}
|f(t)|\leq\epsilon |t|+C_\epsilon |t|^{p-1}\quad\text{and}\quad F(t)\leq\epsilon |t|^2+C_\epsilon |t|^{p},\end{equation}
for all $2<p <2^{*}_s$, where $2^*_s=2N/(N-2s)$.

\begin{lemma}\label{lt2f-F}
For each $t>0$ and $u\in H^{s}(\mathbb{R}^N)$ it holds
\[\frac{t^2}{2}f(u)u-F(tu)\leq \frac{1}{2}f(u)u-F(u).\]
\end{lemma}
\begin{proof}Define $\psi(t)=\frac{t^2}{2}f(u)u-F(tu)$. Then $\psi'(t)=tf(u)u-f(tu)u$, from what follows $\psi'(1)=0$. Since \[\psi'(t)=tu^2\left[\frac{f(u)}{u}-\frac{f(tu)}{tu}\right],\quad u\neq 0,\]
we have $\psi'(t)>0$ if $0<t<1$ and $\psi'(t)<0$ if $t>1$.Thus, $\psi(1)=\displaystyle\max_{t\geq0}\psi(t)>0$ and our claim follows.
$\hfill\Box$\end{proof}

Observe that Definition \ref{defsolution} is satisfied by critical points of the functional
\[\Phi(u)=\frac{1}{2}\int_{ \mathbb{R}^N}\left|(-\Delta+m^2)^{s/2}u\right|^2\dd x-\int_{ \mathbb{R}^N}F(u)\dd x=\frac{1}{2}\|u\|^2-\int_{ \mathbb{R}^N}F(u)\dd x\]
and
\begin{equation}\label{Phi'}
\Phi'(u)\cdot u=\|u\|^2-\int_{\mathbb{R}^n}f(u)u\dd x.
\end{equation}

We denote by $\mathcal{N}$ the Nehari manifold naturally attached to $\Phi$:
\[\mathcal{N}=\left\{u\in H^{s}(\mathbb{R}^N)\setminus\{0\}\,:\,\Phi'(u)\cdot u=0\right\}\]
and by $P(u)$ the functional generated by the Pohozaev-type identity \eqref{Poho}:
\[P(u)=\frac{N-2s}{2}\|u\|^2
+m^2 s\int_{ \mathbb{R}^N}\frac{|\hat{u}(\xi)|^2}{(m^2+4\pi^2|\xi|^2)^{1-s}}\dd\xi-N\int_{ \mathbb{R}^N}F(u)\dd x.\]

Now, for each $u\in H^{s}(\mathbb{R}^N)$ and $t>0$, denote by
\[u_t(x)=tu\left(\frac{x}{t^2}\right)\in  H^{s}(\mathbb{R}^N)\]
and consider  
\begin{align}\label{hu}h_u(t)&=\Phi(u_t)=\frac{1}{2}\int_{ \mathbb{R}^N}\left|(-\Delta+m^2)^{s/2} tu\left(\frac{x}{t^2}\right)\right|^2\dd x-\int_{ \mathbb{R}^N}F\left(tu\left(\frac{x}{t^2}\right)\right)\dd x \nonumber\\
&=\frac{t^{4N +2}}{2}\int_{ \mathbb{R}^N}(m^2+4\pi^2|\xi|^2)^{s}|\hat{u}(t^2\xi)|^2\dd\xi-t^{2N}\int_{ \mathbb{R}^N}F(tu)\dd x \nonumber \\
&=\frac{t^{2N +2 -4s}}{2}\int_{ \mathbb{R}^N}(t^4m^2+4\pi^2|\xi|^2)^{1/2}|\hat{u}(\xi)|^2\dd\xi-t^{2N}\int_{ \mathbb{R}^N}F(tu)\dd x\nonumber\\
&=t^{2N+2}\left[\frac{1}{2}\int_{ \mathbb{R}^N}\left(m^2+\frac{4\pi^2|\xi|^2}{t^4}\right)^{s}|\hat{u}(\xi)| ^2\dd\xi-\int_{ \mathbb{R}^N}\frac{F(tu)}{(tu)^2}u^2\dd x\right].\end{align}

According to ($f_2$), the expression between brackets in \eqref{hu} converges to \[\frac{m^{2s}-k}{2}\int_{ \mathbb{R}^N}|u|^2\dd x,\] from what follows
\begin{equation}\label{hassymp}\lim_{t\to \infty}h_u(t)=-\infty.\end{equation}
But we also have that
\[\frac{h_u(t)}{t^{2N+2}}=\frac{1}{2t^{4s}}\int_{ \mathbb{R}^N}(t^4m^2+4\pi^2|\xi|^2)^{s}|\hat{u}(\xi)|^2\dd\xi-\int_{ \mathbb{R}^N}\frac{F(tu)}{(tu)^2}u^2\dd x,\]
so that ($f_2$) yields $\displaystyle\lim_{t\to 0}\frac{h_u(t)}{t^{N+2}}=\infty$, from what follows that \[h_u(t)>0, \quad\text{if }\ t>0\quad \text{is small enough}.\]

Therefore, $h_u(t)$ attains a maximum point since, for each fixed $u\in H^{s}(\mathbb{R}^N)\setminus\{0\}$, $h_u\in C^1(\mathbb{R}_+,\mathbb{R})$.

Taking the derivative of $h_u$, we obtain
\begin{align}\label{hlinhau}
h'_u(t)&=(N+ 1-2s)t^{2N+1 -4s}\int_{ \mathbb{R}^N}(t^4m^2+4\pi^2|\xi|^2)^{s}|\hat{u}(\xi)|^2\dd\xi\nonumber\\
&\qquad+2st^{2N+5-4s}\int_{ \mathbb{R}^N}\frac{m^2|\hat{u}(\xi)|^2}{(t^4m^2+4\pi^2|\xi|^2)^{1-s}}\dd\xi\nonumber\\
&\qquad-2Nt^{2N-1}\int_{ \mathbb{R}^N}F(tu)\dd x-t^{2N}\int_{ \mathbb{R}^N}f(tu)u\dd x.
\end{align}

So, when $t=1$,
\begin{align}h'_u(1)&=(N+1-2s)\int_{ \mathbb{R}^N}(m^2+4\pi^2|\xi|^2)^{s}|\hat{u}(\xi)|^2\dd\xi \nonumber\\&+2sm^2\int_{ \mathbb{R}^N}\frac{|\hat{u}(\xi)|^2}{(m^2+4\pi^2|\xi|)^{1-s}}\dd\xi-2N\int_{ \mathbb{R}^N}F(u)\dd x-\int_{ \mathbb{R}^N}f(u)u\dd x\nonumber\\
&=\Phi'(u)\cdot u+2P(u)=:J(u).\label{J}\end{align}

This motivates to consider the Nehari-Pohozaev manifold
\[\mathcal{M}=\left\{u\in H^{s}(\mathbb{R}^N)\setminus\{0\}\,:\, J(u)=0\right\}.\]

Changing variables, we observe that
\begin{align}\label{h'1}h'_{u_t}(1)=th'_u(t),\end{align}

It follows from the Pohozaev-type identity \eqref{Poho} that any solution $u\in H^{s}(\mathbb{R}^N)$ of \eqref{prn} belongs to $\mathcal{M}$, since $P(u)=0$ and $\Phi'(u)\cdot u=0$.

Furthermore, taking into accoun \eqref{h'1},
\begin{align}\label{characM}
u_t\in\mathcal{M}\quad\Leftrightarrow\quad J(u_t)=0\quad\Leftrightarrow\quad h'_{u_t}=1\quad\Leftrightarrow\quad h'_{u}(t)=0.
\end{align}
We now show that there exists a unique point $t$ where $h_u(t)$ attains its maximum.
\begin{lemma}
For each $u\in H^{s}(\mathbb{R}^N)\setminus\{0\}$, there exists a unique $t_u=t(u)>0$ such that $h_u(t)$ attains its maximum at $t_u$. The function $h_u(t)$ is positive and increasing for $t\in (0,t_u]$ and decreasing for $t>t_u$.
	
Furthermore, the function
\[u\mapsto t_u\]
is continuous and
\[u_{t_u}\in \mathcal{M}\quad\text{and}\quad\Phi(u_{t_u})=\max_{t>0}\Phi(u_t)>0.\]
\end{lemma}
\begin{proof}We have already shown that $h_u(t)$ attains its maximum at a point $t_u$. Since $h_u\in C^1(\mathbb{R}_+,\mathbb{R})$, we have $h'_u(t_u)=0$.
	
According to \eqref{hlinhau} we have
\begin{subequations}\begin{align}
h'_u(t)=0&\Leftrightarrow
t^{2N+1}\left[\frac{(N+1-2s)}{t^{4s}}\int_{\mathbb{R}^N}\left(t^4m^2+4\pi^2|\xi|^2\right)^{s}|\hat{u}(\xi)|^2\dd\xi\right.\nonumber\\
&\quad\quad\left.+2sm^2\int_{\mathbb{R}^N}\frac{t^{4(1-s)}|\hat{u}(\xi)|^2}{(m^2t^4+4\pi^2|\xi|^2)^{1-s}}\dd\xi\right]\nonumber\\
&\qquad-t^{2N+1}\left[2N\int_{\mathbb{R}^2}\frac{F(tu)}{t^2}\dd x+\int_{\mathbb{R}^N}\frac{f(tu)u}{t}\dd x\right]=0\label{hlinhau2a}\\
&\Leftrightarrow \frac{(N+1-2s)}{t^{4s}}\int_{\mathbb{R}^N}\left(t^4m^2+4\pi^2|\xi|^2\right)^{s}|\hat{u}(\xi)|^2\dd\xi\nonumber\\
&\qquad+2sm^2\int_{\mathbb{R}^N}\frac{t^{4(1-s)}|\hat{u}(\xi)|^2}{(m^2t^4+4\pi^2|\xi|^2)^{1-s}}\dd\xi\nonumber\\
&\qquad -\frac{1}{t^2}\int_{\mathbb{R}^N}\left[2NF(tu)+f(tu)tu\right]\dd x=0\label{hlinhau2}
\end{align}\end{subequations}
	
We denote
\[I_1(t)=\frac{1}{t^2}\int_{\mathbb{R}^N}2NF(tu)\dd x,\quad I_2(t)=\frac{1}{t^2}\int_{\mathbb{R}^N}f(tu)tu\dd x\]
and
\begin{align*}g(t)&=\frac{(N+ 1-2s)}{t^{4s}}\int_{\mathbb{R}^N}\left(t^4m^2+4\pi^2|\xi|^2\right)^{s}|\hat{u}(\xi)|^2\dd\xi\\
&\qquad +2sm^2\int_{\mathbb{R}^N}\frac{t^{4(1-s)}|\hat{u}(\xi)|^2}{(m^2t^4+4\pi^2|\xi|^2)^{1-s}}\dd\xi.\end{align*}
	
It follows from $(f_1)$ that
\begin{align*}
\frac{\dd}{\dd t}I_1(t)
&=\frac{2N}{t^3}\int_{\mathbb{R}^N}\left[f(tu)tu-2F(tu)\right]\dd x>0
\end{align*}
and also
\[\frac{\dd}{\dd t}I_2(t)=\frac{\dd}{\dd t}\int_{\mathbb{R}^N}\frac{f(tu)}{tu}|u|^2\dd\xi=\int_{\mathbb{R}^N}\frac{\dd}{\dd t}\left(\frac{f(tu)}{tu}\right)|u|^2\dd x>0.\]
	
We conclude that $I_1(t)+I_2(t)$ is a strictly increasing function.
	
We will now show that $g(t)$ is strictly decreasing. In fact,

\begin{align*}
g'(t)
&=\frac{-4s(N+1 -2s)}{t^{4s+1}}\int_{ \mathbb{R}^N}(t^4m^2+4\pi^2|\xi|^2)^{s}|\hat{u}(\xi)|^2\dd\xi\\
&\quad +4s(N+1 -2s)\int_{ \mathbb{R}^N}\frac{t^{3-4s}m^2|\hat{u}(\xi)|^2}{(t^4m^2+4\pi^2|\xi|^2)^{1-s}}\dd\xi\\
&\quad+2sm^24(1-s)t^{4(1-s)-1}\int_{ \mathbb{R}^N}\frac{|\hat{u}(\xi)|^2}{(t^4m^2+4\pi^2|\xi|^2)^{1-s}}\dd\xi\\
&\quad-2sm^4 4(s-1)t^{4(1-s) + 3}\int_{ \mathbb{R}^N}\frac{|\hat{u}(\xi)|^2}{(t^4m^2+4\pi^2|\xi|^2)^{2 -s}}\dd\xi\\
&=\frac{-4s(N+1 -2s)}{t^{4s +1}}\int_{\mathbb{R}^N}\left[\left(t^4m^2+4\pi^2|\xi|^2\right)^{s}-\frac{m^2t^4}{(t^4m^2+4\pi^2|\xi|^2)^{1-s}}\right]|\hat{u}(\xi)|^2\dd\xi\\
&\quad+2sm^2 4(1-s)t^{4(1-s)-1}\int_{\mathbb{R}^N}\frac{|\hat{u}(\xi)|^2}{(t^4m^2+4\pi^2|\xi|^2)^{1-s}}\left[1-\frac{t^4m^2}{t^4m^2+4\pi^2|\xi|^2}\right]\dd\xi\\
&=\frac{-4s(N+1 -2s)}{t^{4s+1}}\int_{\mathbb{R}^N}\frac{4\pi^2|\xi|^2|\hat{u}(\xi)|^2}{(t^4m^2+4\pi^2|\xi|^2)^{1-s}}\dd\xi\\
&\quad+2sm^2 4(1-s)t^{4(1-s)-1}\int_{\mathbb{R}^N}\frac{4\pi^2|\xi|^2|\hat{u}(\xi)|^2}{\left(t^4m^2+4\pi^2|\xi|^2\right)^{2-s}}\dd\xi\\
&=\frac{-4s}{t^{4s +1}}\int_{\mathbb{R}^N}\frac{4\pi^2|\xi|^2|\hat{u}(\xi)|^2}{(t^4m^2+4\pi^2|\xi|^2)^{1-s}}\left[(N+1-2s)-2(1-s)\frac{t^4m^2}{t^4m^2+4\pi^2|\xi|^2}\right]\dd\xi\\
&<0.
\end{align*}
	
Thus, $g(t)-(I_1(t)+I_2(t))$ is strictly decreasing, proving the uniqueness of $t_u$.
	
To prove that the function $u\mapsto t_u$ is continuous, let us consider a sequence $(u_n)$ such that $u_n\to u$ in $H^{s}(\mathbb{R}^N)$. We denote $t_n=t_{u_n}$. We claim that $(t_n)$ is bounded.
	
To prove our claim, we observe that \eqref{hlinhau2} implies that	\begin{align}\label{hlinhau3}
h'_{u_n}(t_n)=0&\Leftrightarrow \frac{(N+1-2s)}{t^{4s}_n}\int_{\mathbb{R}^N}\left(t^4_nm^2+4\pi^2|\xi|^2\right)^s|\hat{u}_n(\xi)|^2\dd\xi\nonumber\\&\qquad+2sm^2\int_{\mathbb{R}^N}\frac{t^{4(1-s)}_n|\hat{u}_n(\xi)|^2}{(m^2t^4_n+4\pi^2|\xi|^2)^{1-s}}\dd\xi\nonumber\\
&\qquad -\int_{\mathbb{R}^N}\left[2N\frac{F(t_nu_n)}{t^2_n}+\frac{f(t_nu_n)u_n}{t_n}\right]\dd x=0.
\end{align}
	
Since $t_n>0$ for all $n$, suppose that $t_n\to\infty$. Application of the dominated convergence theorem and hypotheses $(f_1)$ and $(f_2)$ yield
\begin{align*}
\int_{ \mathbb{R}^N}\left[2N\frac{F(t_nu_n)}{t^2_n}+\frac{f(t_nu_n)}{t_n}u_n\right]\dd x&\phantom{\hspace*{.2cm}}=\phantom{\hspace*{.1cm}}\int_{ \mathbb{R}^N}\left[2N\frac{F(t_nu_n)}{|t_nu_n|^2}+\frac{f(t_nu_n)}{|t_nu_n|}\right]|u_n|^2\dd x\\
&\stackrel{n\to\infty}{\longrightarrow}\int_{ \mathbb{R}^N}\left[2N\frac{k}{2}+k\right]|u|^2\dd x\\
&\qquad\quad=(N+1)k\int_{ \mathbb{R}^N}|u|^2\dd x.
\end{align*}
	
On the other side, since
\[\frac{(N+1-2s)}{t^{4s}_n}\int_{\mathbb{R}^N}\left(t^4_nm^2+4\pi^2|\xi|^2\right)^{s}|\hat{u}_n(\xi)|^2\dd\xi+2sm^2\int_{\mathbb{R}^N}\frac{t^{4(1-s)}_n|\hat{u}_n(\xi)|^2}{(m^2t^4_n+4\pi^2|\xi|^2)^{1-s}}\dd\xi\]
converges to
\[\int_{\mathbb{R}^N}(N+1)m^{2s}|\hat{u}(\xi)|^2\dd\xi=\int_{\mathbb{R}^N}(N+1)m^{2s}|u(x)|^2\dd x,\]
it follows from \eqref{hlinhau3} that
\[(N+1)(m^{2s}-k)\int_{ \mathbb{R}^N}|u|^2\dd x=0,\]
and we have reached a contradiction.
	
Thus, we can suppose that $t_n\to t_0\in (0,\infty)$. (Observe that we already know that $t_0\neq 0$.) By applying once again the dominated convergence theorem to \eqref{hlinhau3}, we obtain both
\[\int_{\mathbb{R}^N}\left[2N\frac{F(t_nu_n)}{t^2_n}+\frac{f(t_nu_n)u_n}{t_n}\right]\dd x\stackrel{n\to\infty}{\longrightarrow}\int_{\mathbb{R}^N}\left[2N\frac{F(t_0u)}{t^2_0}+\frac{f(t_0u)u}{t_0}\right]\dd x\]
and
\begin{multline*}
\frac{(N+1-2s)}{t^{4s}_n}\int_{\mathbb{R}^N}\left(t^4_nm^2+4\pi^2|\xi|^2\right)^{s}|\hat{u}_n(\xi)|^2\dd\xi+2sm^2\int_{\mathbb{R}^N}\frac{t^{4(1-s)}_n|\hat{u}_n(\xi)|^2}{(m^2t^4_n+4\pi^2|\xi|^2)^{1-s}}\dd\xi\\
\stackrel{n\to\infty}{\rightarrow}\!\int_{\mathbb{R}^N}\frac{(N+1-2s)\left(t^4_0m^2+4\pi^2|\xi|^2\right)^{s}|\hat{u}(\xi)|^2}{t^{4s}_0}\dd\xi+2sm^2\int_{\mathbb{R}^N}\frac{t^{4(1-s)}_0|\hat{u}(\xi)|^2}{(m^2t^4_0+4\pi^2|\xi|^2)^{1-s}}\dd\xi
\end{multline*}
Thus, by passing to the limit in \eqref{hlinhau2a} yields
\begin{align*}
&t^{2N+1}_0\left[(N+1-2s)\int_{\mathbb{R}^N}\frac{\left(t^4_0m^2+4\pi^2|\xi|^2\right)^{s}|\hat{u}(\xi)|^2}{t^{4s}_0}\dd\xi\right.\\
&\left.+2sm^2\int_{\mathbb{R}^N}\frac{t^{4(1-s)}_0|\hat{u}(\xi)|^2}{(m^2t^4_0+4\pi^2|\xi|^2)^{1-s}}\dd\xi\right]
-t^{2N+1}_0\left[2N\int_{\mathbb{R}^2}\frac{F(t_0u)}{t^2_0}\dd x+\int_{\mathbb{R}^N}\frac{f(t_0u)u}{t}
\dd x\right]\\
&=0,
\end{align*}
and it follows from \eqref{hlinhau2a} that $h'_u(t_0)=0$. Uniqueness of $t_u$ imply $t_u=t_0$ and so $t_n\to t_u$. We are done.
$\hfill\Box$\end{proof}

\begin{lemma}\label{Phi-cJ}
For all $u\in H^{s}(\mathbb{R}^N)\setminus \{0\}$ we have
\[\Phi(u)-\frac{1}{2N+2}J(u)>0.\]
\end{lemma}
\begin{proof}We have
\begin{multline*}
\Phi(u)-\frac{1}{2N+2}J(u)=\hfill
\end{multline*}
\begin{align}\label{eqPhi-cJ}
&=\frac{1}{2}\|u\|^2-\int_{ \mathbb{R}^N}F(u)\dd x-\frac{(N+1-2s)}{2N+2}\|u\|^2-\frac{2sm^2}{2N+2}\int_{ \mathbb{R}^N}\frac{|\hat{u}(\xi)|^2}{(m^2+4\pi^2|\xi|^2)^{1-s}}\dd\xi\nonumber\\
&\qquad +\frac{2N}{2N+2}\int_{\mathbb{R}^N}F(u)\dd x+\frac{1}{2N+2}\int_{ \mathbb{R}^N}f(u)u\dd x\nonumber\\
&=\frac{2s}{2N+2}\left[\int_{\mathbb{R}^N}(m^2+4\pi^2|\xi|^2)^{s}|\hat{u}(\xi)|^2\dd\xi-\int_{ \mathbb{R}^N}\frac{m^2|\hat{u}(\xi)|^2}{(m^2+4\pi^2|\xi|^2)^{1-s}}\dd\xi\right]\nonumber\\
&\qquad +\frac{1}{2N+2}\int_{ \mathbb{R}^N}\left[f(u)u-2F(u)\right]\dd x\nonumber\\
&= \frac{2s}{2N+2}\left[\int_{ \mathbb{R}^N}\frac{4\pi^2|\xi|^2|\hat{u}(\xi)|^2}{(m^2+4\pi^2|\xi|^2)^{1-s}}\dd\xi+\int_{ \mathbb{R}^N}\left[f(u)u-2F(u)\right]\dd x\right]\\
&>\frac{2s}{2N+2}\int_{ \mathbb{R}^N}\frac{4\pi^2|\xi|^2|\hat{u}(\xi)|^2}{(m^2+4\pi^2|\xi|^2)^{1-s}}\dd\xi> 0,\ \ \text{if }\ u\neq 0.\nonumber
\end{align}
	
\vspace*{-.4cm}$\hfill\Box$\end{proof}

\begin{lemma}\label{manifold}
If $f$ satisfies hypotheses $(f_1)$ and $(f_2)$, then
\begin{enumerate}
\item [$(i)$] $J(u)>0$ for all $0<\|u\|\leq\rho$ and $\mathcal{M}$ is a closed subset of $H^{s}(\mathbb{R}^N)$;
\item  [$(ii)$] For all $u\in \mathcal{M}$ we have $\Phi(u)>0$;
\item  [$(iii)$] $\mathcal{M}$ is a $C^1$-manifold.
\end{enumerate}
\end{lemma}
\begin{proof}Since there exist constants $\gamma_p$ such that $|u|_p\leq \gamma_p\|u\|$ for all $2<p<2^{*}_s$, it follows from the definition of $J$ - see \eqref{J} - and \eqref{fandF} that
\begin{align*}
J(u)&\geq (N+1-2s)\|u\|^2-2N\int_{ \mathbb{R}^N}F(u)\dd x-\int_{ \mathbb{R}^N} f(u)u\dd x\\
&\geq (N+1-2s)\|u\|^2-(2N+1)\left[\epsilon |u|^2_2+C_\epsilon |u|^2_p\right]\\
&\geq  \left[(N+1-2s)-(2N+1)\gamma^2_2\epsilon\right]\|u\|^2-(2N+1)C_\epsilon\gamma^p_p\|u\|^p\\
&=\frac{(N+1-2s)}{2}\|u\|^2-(2N+1)C_\epsilon\gamma^p_p\|u\|^p,
\end{align*}
if we choose $\epsilon=(N+1-2s)/(2(2N+1)\gamma^2_2)$. Now, taking
\[\rho=\|u\|=\left(\frac{(N+1-2s)}{p(2N+1)C_\epsilon\gamma^p_p}\right)^{1/(p-2)},\]
we obtain that
\[J(u)\geq \frac{(N+1-2s)(p-2)}{2p}\rho^2>0\]
for all $0<\|u\|\leq\rho$. Thus, $u=0$ is an isolated point of $J^{-1}(0)$ and $\mathcal{M}\subset \mathcal{M}\cup \{0\}=J^{-1}(0)$ is closed.
	
For all $u\in\mathcal{M}$ we have $J(u)=0$, so that
\[\Phi(u)=\Phi(u)-\frac{1}{2N+2}J(u)\]
and ($ii$) follows from Lemma \eqref{Phi-cJ}.
	
Since
\begin{align*}
J'(u)\cdot u&=2(N+1-2s)\|u\|^2+4sm^2\int_{ \mathbb{R}^N}\frac{|\hat{u}(\xi)|^2}{(m^2+4\pi^2|\xi|^2)^{1-s}}\dd\xi\\
&\qquad -2N\int_{ \mathbb{R}^N}f(u)u\dd x-\int_{ \mathbb{R}^N}\left[f'(u)u^2+f(u)u\right]\dd x,
\end{align*}
substitution of $J(u)=0$ in the last equation and using hypothesis ($f_1$) we obtain
\begin{align*}
J'(u)\cdot u&=-2N\int_{ \mathbb{R}^N}\left[f(u)u-2F(u)\right]\dd x-\int_{ \mathbb{R}^N}\left[f'(u)u^2-f(u)u\right]\dd x\\
&<0,
\end{align*}
as consequence of our hypotheses. We are done.	
$\hfill\Box$\end{proof}

We now define the minimax value
\[c=\inf_{\gamma\in\Gamma}\max_{0\leq t\leq 1}\Phi(\gamma(t)),\]
where
\[\Gamma=\left\{\gamma\in C([0,1],H^{s}(\mathbb{R}^N))\,:\,\gamma(0)=0\ \ \text{and}\ \ \Phi(\gamma(1))<0\right\},\]
and the infimum in the Nehari-Pohozaev manifold
\[\tilde{c}=\inf_{u\in\mathcal{M}}\Phi(u)=\inf_{u\in H^{s}}\max_{t>0}\Phi(u_t)\geq 0.\]
\begin{lemma}\label{cctilde}
The level $c$ is well-defined and $c=\tilde{c}$.
\end{lemma}
\begin{proof}Taking into account \eqref{hassymp}, for all $u\in H^{s}(\mathbb{R}^N)\setminus\{0\}$ there exists $t_1=t_1(u)$ such that $\Phi(u_{t_1})<0$, where $u_t(x)=tu(x/t^2)$. Defining $\gamma_0(t)=u_{tt_1}$, if $t>0$ and $\gamma_0(0)=0$, we have
\[\Phi(\gamma_0(1))=\Phi(u_{t_1})<0,\]
from what follows that $\gamma_0\in\Gamma$.
	
Furthermore,
\[\max_{t\geq 0}\Phi(u_{tt_1})\geq\max_{0\leq t\leq q}\Phi(u_{tt_1})\geq \inf_{\gamma\in\Gamma}\max_{0\leq t\leq 1}\Phi(\gamma(t))=c,\]
proving that $\tilde{c}\geq c$.
	
But we also know the existence of $\rho>0$ such that $J(u)\geq 0$ for all $u\in H^{s}(\mathbb{R}^N)$, if $\|u\|\leq \rho$. Lemma \ref{Phi-cJ} yields 	
\[\Phi(u)\geq \frac{1}{2N+2}J(u)\geq 0,\quad\text{if }\ \|u\|\leq\rho.\]
	
If $\gamma\in \Gamma$, since $\Phi(\gamma(1))<0$, a new application of Lemma \ref{Phi-cJ} gives
\[J(\gamma(1))\leq (2N+2)\Phi(\gamma(1))<0,\]
so that $J(\gamma(0))=0$, $J(\gamma(t))>0$ if $\|\gamma(t)\|<\rho$ and $J(\gamma(1))<0$. We conclude the existence of $\tilde{t}$ such tant $J(\gamma(\tilde{t}))=0$, proving that $\gamma$ intercepts $\mathcal{M}$. Therefore,
\[\max_{0\leq t\leq 1}\Phi(\gamma(t))\geq \inf_{u\in\mathcal{M}}\Phi(u)=\tilde{c},\]
so that $c\geq \tilde{c}$ and completing the proof of $c=\tilde{c}$.	
$\hfill\Box$\end{proof}

\begin{definition}
A sequence $(u_n)\in H^{s}(\mathbb{R}^N)$ is a Cerami sequence for $\Phi$ in the level $\theta$ if
\[\Phi(u_n)\to\theta\quad\text{and}\quad \|\Phi'(u_n)\|_*\left(1+\|u_n\|\right)\to 0, \]
where $\|\cdot\|_*$ stands for the norm in $\left(H^{s}(\mathbb{R}^N)\right)^*$.
\end{definition}

\begin{lemma}\label{Cbounded}
Let $(u_n)$ be a Cerami sequence for $\Phi$ at the level $\theta>0$. Then, passing to a subsequence if necessary, $(u_n)$ is bounded in $H^{s}(\mathbb{R}^N)$.
\end{lemma}
\begin{proof}Since $\|\Phi'(u_n)\|_*(1+\|u_n\|)\to 0$, we have
\[\|\Phi'(u_n)\|_*\leq \|\Phi'(u_n)\|_*(1+\|u_n\|)\leq \frac{1}{n}\]
for $n$ big enough. Thus, we can suppose that
\begin{align}\label{Phi'ineq}
-\frac{1}{n}<\Phi'(u_n)\cdot u_n=\|u_n\|^2-\int_{\mathbb{R}^n}f(u_n)u_n\dd x<\frac{1}{n}.
\end{align}
	
The last inequality and Lemma \ref{lt2f-F} imply that
\begin{align}\label{ineqPhi}
\Phi(tu_n)&=\frac{t^2}{2}\|u_n\|^2-\int_{\mathbb{R}^N}F(tu_n)\dd x \nonumber\\
&\leq \frac{t^2}{2}\left[\frac{1}{n}+\int_{\mathbb{R}^N}f(u_n)u_n\dd x\right]-\int_{\mathbb{R}^N}F(tu_n)\dd x\\
&\leq \frac{t^2}{2n}+\int_{\mathbb{R}^N}\left[\frac{1}{2}f(u_n)u_n-F(u_n)\right]\dd x.\nonumber
\end{align}
	
But it also follows from \eqref{Phi'ineq} that
\begin{align}\label{Phiineq}
\Phi(u_n)&=\frac{1}{2}\|u_n\|^2-\int_{\mathbb{R}^n}F(u_n)\dd x\nonumber\\
&\geq -\frac{1}{2n}+\int_{\mathbb{R}^n}\left[\frac{1}{2}f(u_n)u_n-F(u_n)\right]\dd x.
\end{align}
	
Thus, it follows from \eqref{ineqPhi} and \eqref{Phiineq},
\begin{align}\label{Phitun}
\Phi(tu_n)\leq \frac{t^2}{2n}+\left[\frac{1}{2n}+\Phi(u_n)\right].
\end{align}
	
Since $\Phi(u_n)=\theta+O_n(1)$, \eqref{Phitun} yields
\begin{equation}\label{ineqPhitun}
\frac{t^2}{2}\|u_n\|^2\leq \frac{t^2}{2n}+\frac{1}{2n}+\theta+O_n(1)+\int_{\mathbb{R}^N}F(tu_n)\dd x.
\end{equation}
	
Taking $t_n=2\sqrt{\theta}/\|u_n\|$ and substituting $t_n$ into \eqref{ineqPhitun}, yields
\begin{align*}
2\theta&\leq \frac{4\theta}{n\|u_n\|^2}+O_n(1)+\theta+\int_{\mathbb{R}^N}F(tu_n)\dd x\nonumber\end{align*}
and so
\begin{align}\label{Lions1}\theta &\leq \frac{4\theta}{n\|u_n\|^2}+ O_n(1)+4\epsilon\theta\int_{\mathbb{R}^N}\left(\frac{u_n}{|u_n|}\right)^2\dd x+C_\epsilon(2\sqrt{\theta})^p\int_{\mathbb{R}^N}\left(\frac{u_n}{|u_n|}\right)^p\dd x.
\end{align}
	
Now, by contradiction, suppose that $\|u_n\|\to\infty$ for a subsequence and consider the bounded sequence
\[\tilde{u}_n=\frac{u_n}{\|u_n\|}.\]
	
Since $H^{s}(\mathbb{R}^N)$ is reflexive, passing to a subsequence if necessary, we can suppose $\tilde{u}_n\rightharpoonup \tilde{u}$ for some $\tilde{u}\in H^{s}(\mathbb{R}^N)$. There are two possible cases:
\begin{enumerate}
\item [$(i)$] $\displaystyle\limsup_{n\to\infty}\sup_{y\in\mathbb{R}^N}\int_{B_1(y)}|\tilde{u}_n|^2\dd x=0$;
\item [$(ii)$] $\displaystyle\limsup_{n\to\infty}\sup_{y\in\mathbb{R}^N}\int_{B_1(y)}|\tilde{u}_n|^2\dd x>0$.
\end{enumerate}
	
If case ($i$) occurs, then $\tilde{u}_n\to 0$ in $L^p(\mathbb{R}^N)$, if $2<p<2^{*}_s$ by a well-known result by Lions (see \cite[Lemma 1.21]{Willem}).
Thus, it follows from \eqref{Lions1} that
\begin{align}\theta&\leq O_n(1)+4\epsilon\theta\int_{ \mathbb{R}^N}|\tilde{u}_n|^2\dd x\leq O_n(1)+\frac{4\epsilon\theta}{m^{2s}}\int_{ \mathbb{R}^N}(m^2+4\pi^2|\xi|^2)^{s}|\,\mathcal{F}(\tilde{u}_n)|^2\dd \xi\nonumber\\
&\leq O_n(1)+\frac{4\epsilon\theta}{m}\|\tilde{u}_n\|^2=O_n(1)+\frac{4\epsilon\theta}{m^{2s}}
\end{align}
and we have reached a contradiction by taking $\epsilon=m^{2s}/8$.
	
Now suppose that  case ($ii$) occurs. If $\delta=\displaystyle\limsup_{n\to\infty}\sup_{y\in\mathbb{R}^N}\int_{B_1(y)}|\tilde{u}_n|^2\dd x>0$, passing to a subsequence if necessary, we have
\[\int_{B_1(y)}|\tilde{u}_n|^2\dd x>\frac{\delta}{2}.\]
Therefore, there exists a sequence $(y_n)$ such that, for all $n\in\mathbb{N}$,
\[\int_{B_1(y_n)}|\tilde{u}_n|^2\dd x>\frac{\delta}{2}>0.\]
	
We define
\[\tilde{v}_n(x)=\tilde{u}_n(x+y_n).\]
Since $\tilde{v}_n$ is a translation of $\tilde{u}_n$, we have $\|\tilde{v}_n\|=1$. Thus, passing to a subsequence we can suppose that
\[\tilde{v}_n\rightharpoonup v\ \ \text{in }\ H^{s}(\mathbb{R}^N),\quad \tilde{v}_n\to \tilde{v}\ \ \text{in }\ L^2_{loc}(\mathbb{R}^N)\quad \text{and}\quad \tilde{v}_n(x)\to \tilde{v}(x)\ \ \text{a.e. in }\ \mathbb{R}^N.\]
	
We now consider two cases: $(y_n)$ unbounded and $(y_n)$ bounded. In the first case, since
\[\frac{\delta}{2}<\int_{\bar{B}_1(y_n)}|\tilde{u}_n|^2\dd x=\int_{\bar{B}_1(0)}|\tilde{v}_n|^2\dd x\to \int_{\bar{B}_1(0)}|\tilde{v}|^2\dd x,\]
we conclude that $\tilde{v}\neq 0$.  Thus, there exists $\Omega\subset B_1(0)$, such that $|\tilde{v}(x)|>0$ for all $x\in \Omega$, with $\Omega$ satisfying $|\Omega|>0$. (Observe that $\Omega$ does not depend on $n$.)
	
In the second case, suppose that $|y_n|\leq R$ for all $n$. We can suppose that $R>1$. Thus,
\[\frac{\delta}{2}<\int_{B_1(0)}|\tilde{u}_n(x+y_n)|^2\dd x\leq \int_{B_{2R}(0)}|\tilde{u}_n(x+y_n)|^2\dd x\to \int_{B_{2R}(0)}|\tilde{v}(x)|^2\dd x,\]
since $\tilde{v}_n\to \tilde{v}$ in $L^2(B_{2R}(0))$. So, as before, we conclude the existence of $\Omega\subset B_{2R}(0)$ such that $\tilde{v}>0$ in $\Omega$.
	
In both cases, seeing that
\[0<|\tilde{v}(x)|=\lim_{n\to\infty}\frac{|u_n(x+y_n)|}{\|u_n\|},\ \ \forall\ x\in\Omega,\]
we conclude that
\[\lim_{n\to\infty}|u_n(x+y_n)|=\infty,\ \ \text{if }\ x\in\Omega.\]
	
By applying hypothesis ($f_3$) and Fatou's lemma, we have
\begin{multline*}\liminf_{n\to\infty}\int_{ \mathbb{R}^N}\left[(1/2)f(u_n(x+y_n))u_n(x+y_n)-F(u_n(x+y_n))\right]\dd x\\
\geq \liminf_{n\to\infty}\int_{ \Omega}\left[(1/2)f(u_n(x+y_n))u_n(x+y_n)-F(u_n(x+y_n))\right]\dd x=\infty.\end{multline*}
	
But
\[\int_{ \mathbb{R}^N}\left[(1/2)f(u_n)u_n-F(u_n)\right]\dd x=\Phi(u_n)-\frac{1}{2}\Phi'(u_n)= \theta+O_n(1)\]
and once again we reached a contradiction, and we are done.
$\hfill\Box$\end{proof}

The existence of a Cerami sequence for $\Phi$ at the level $c$ is a consequence of the Ghoussoub-Preiss theorem, that we now recall, for the convenience of the reader, using our notation. A good exposition of this result can be found in one of Ekeland's books, see \cite[Theorem 6, p. 140]{Ekeland}, see also \cite{Ghoussoub}. 
\begin{theorem}[Ghoussoub-Preiss]Let $X$ be a Banach space and $\Phi\colon X\to \mathbb{R}$ a continuous, Gateaux-differentiable function, such that $\Phi'\colon X\to X$ is continuous from the norm topology of $X$ to the weak${^*}$ topology of $X^*$. Take two points $z_0, z_1$ in $X$ and consider the set $\Gamma$ of all continuous paths from $z_0$ to $z_1$:
\[\Gamma=\left\{\gamma\in C([0,1],X)\,:\, \gamma(0)=z_0,\ \gamma(1)=z_1\right\}.\]
	
Define a number $c$ by
\[c:=\inf_{\gamma\in\Gamma}\max_{0\leq t\leq 1}\Phi(\gamma(t)).\]
	
Assume that there is a closed subset $\mathcal{M}$ of $X$ such that
\[\mathcal{M}\cap \Phi_c\ \emph{ separates }\ z_0\ \ \text{and }\ z_1,\]
with $\Phi_c=\{x\in X\,:\, \Phi(x)\geq c\}$.
	
Then, there exists a sequence $(x_n)$ in $X$ such that
\begin{enumerate}
\item [$(i)$] $\textup{dist}\,(x_n,\mathcal{M})\to 0$;
\item [$(ii)$] $\Phi(x_n)\to c$;
\item [$(iii)$] $(1+\|x_n\|)\|\Phi'(x_n)\|_*\to 0$.
\end{enumerate}
\end{theorem}

In the original Ghoussoub-Preiss theorem, we have $\delta(x_n,\mathcal{M})\to 0$, where $\delta$ stands for the geodesic distance. If $x_0=0$, then $\delta(0,x)=\ln (1+\|x\|)$ and we can change $\delta$ by dist, see \cite[p. 138]{Ekeland}. A closed subset $\mathcal{F}\subset X$ \emph{separates} two points $z_0$ and $z_1$ in $X$ if $z_0$ and $z_1$ belong to disjoint connected components in $X\setminus\mathcal{F}$

Of course, in our case $X=H^{s}(\mathbb{R}^N)$.  If we take $z_0=0$ and $z_1$ such that $\Phi(z_1)<0$, then \[H^{s}(\mathbb{R}^N)\setminus\mathcal{M}=\{0\}\cup\{u\in H^{s}(\mathbb{R}^N):J(u)>0\}\cup\{u\in H^{s}(\mathbb{R}^N):J(u)<0\}\] (remember that $0\notin \mathcal{M}$ and $J(0)=0$).  According to Lemma \ref{manifold}, $B_\rho(0)$ belongs to a connected component of $\{0\}\cup \{u\in H^{s}(\mathbb{R}^N):J(u)>0\}$. Since $\Phi(z_1)<0$, it follows from Lemma \ref{Phi-cJ} that $J(z_1)<0$. Thus, $\mathcal{M}$ separates $z_0$ and $z_1$. But we also have $\mathcal{M}\cap \Phi_c=\mathcal{M}$, since $ \displaystyle\inf_{u\in\mathcal M}\Phi(u)=c$, as consequence of Lemma \ref{cctilde}. So, the assumptions of the Ghoussoub-Preiss theorem are fulfilled.

\textbf{Proof of Theorem \ref{texistence}} Since the Ghoussoub-Preiss theorem guarantees the existence of a Cerami sequence $\{u_n\} \subset H^{s}(\mathbb{R}^N)$, which is bounded by Lemma \ref{Cbounded}, we can suppose that
\[u_n\rightharpoonup u\ \text{ in }\ H^{s}(\mathbb{R}^N),\quad u_n(x)\to u(x)\ \text{a.e. and } u_n\to u\ \text{in }\ L^p_{loc}(\mathbb{R}^N)\]
for $p\in [2,2^{*}_s)$.

We define, mimicking the proof of Lemma \ref{Cbounded},
\[\delta=\limsup_{n\to\infty}\sup_{y\in\mathbb{R}^N}\int_{B_1(y)}|u_n|^2\dd x.\]

If $\delta=0$, by the principle of concentration-compactness of Lions we have $u_n\to 0$ in $L^p(\mathbb{R}^N)$ for any $p\in (2,2^{*}_s)$. Since $(u_n)$ is bounded in $L^2(\mathbb{R}^N)$, there exists $M_0>0$ such that $|u_n|_2\leq M$ for all $n\in\mathbb{N}$. Thus, for any $\eta>0$, by taking $\epsilon=\eta/M_0$, we conclude that
\[\int_{ \mathbb{R}^N}|F(u_n)|\dd x\leq \epsilon |u_n|^2_2+C_\epsilon |u_n|^p_p=\eta+C_\epsilon |u_n|^p_p\]
and $|u_n|_p\to 0$ implies that
\[\int_{ \mathbb{R}^N}F(u_n)\dd x\to 0,\ \text{when }\ n\to\infty.\]

Similarly,
\[\int_{ \mathbb{R}^N}f(u_n)u_n\dd x\to 0,\ \text{ when }\ n\to \infty,\]
and, since $\|u_n\|^2-\int_{ \mathbb{R}^N}f(u_n)u_n=\Phi'(u_n)\cdot u_n\to 0$ if $n\to \infty$, we have $\|u_n\|^2\to 0$.

Thus,
\[0<\tilde{c}=\lim_{n\to\infty}\Phi(u_n)=\lim_{n\to\infty}\left[\frac{1}{2}\|u_n\|^2-\int_{ \mathbb{R}^N}F(u_n)\dd x\right]=0,\]
a contradiction.

If, however, $\delta>0$, there exists a sequence $(y_n)$ such that, for all $n\in\mathbb{N}$,
\begin{align}\label{wneq0}\int_{B_1(y_n)}|\tilde{u}_n|^2\dd x>\frac{\delta}{2}>0.\end{align}

We define $w_n=u_n(x+y_n)$. Then $\|w_n\|=\|u_n\|$, $J(w_n)=J(u_n)$, $\Phi(w_n)=\Phi(u_n)$ and $\Phi'(w_n)\to 0$ when $n\to\infty$. Passing to a subsequence we can suppose that, for $p\in [2,2^{*}_s)$, we have
\[w_n\rightharpoonup w\ \ \text{in }\ H^{s}(\mathbb{R}^N),\quad w_n\to w\ \ \text{in }\ L^p_{loc}(\mathbb{R}^N)\quad \text{and}\quad w_n(x)\to w(x)\ \ \text{a.e. in }\ \mathbb{R}^N.\]

From \eqref{wneq0} follows $w\neq 0$. Furthermore, for all $\varphi\in H^{s}(\mathbb{R}^N)$, we have
\begin{align*}\Phi'(w)\cdot\varphi&=\lim_{n\to\infty}\left[\int_{\mathbb{R}^n}(-\Delta+m^2)^{s/2}w_n(-\Delta+m^2)^{s/2}\varphi\dd x-\int_{ \mathbb{R}^N}f(w_n)\varphi\dd x\right]\\
&=\lim_{n\to\infty}\Phi'(w_n)\cdot\varphi=0,\end{align*}
and we conclude that
\[\Phi'(w)=0,\]
that is, $w$ is a weak solution of \eqref{prn} and, therefore,
satisfies the Pohazaev identity. But
\[J(w)=\Phi'(w)\cdot w+2P(w)=0\]
proves that $w\in\mathcal{M}$. Therefore, $\Phi(w)\geq \tilde{c}$, as consequence of Lemma \ref{cctilde}.

However, by applying Fatou's Lemma to \eqref{eqPhi-cJ} with $w$ instead of $u$, we obtain
\begin{align*}
\Phi(w)&=\Phi(w)-\frac{1}{2N+2}J(w)\\
&=\frac{2s}{2N+2}\left[\int_{ \mathbb{R}^N}\frac{4\pi^2|\xi|^2|\hat{w}(\xi)|^2}{(m^2+4\pi^2|\xi|^2)^{1-s}}\dd\xi+\int_{ \mathbb{R}^N}\left[f(w)w-2F(w)\right]\dd x\right]\\
&\leq \liminf_{n\to\infty}\frac{2s}{2N+2}\left[\int_{ \mathbb{R}^N}\frac{4\pi^2|\xi|^2|\hat{w_n}(\xi)|^2}{(m^2+4\pi^2|\xi|^2)^{1-s}}\right.\dd\xi\\
&\qquad\left.+\int_{ \mathbb{R}^N}\left[f(w_n)w_n-2F(w_n)\right]\dd x\right]\\
&=\liminf_{n\to\infty}\left[\Phi(w_n)-\frac{1}{2N+2}J(w_n)\right]=\liminf_{n\to\infty}\Phi(w_n)=\tilde{c}.
\end{align*}
Thus, we have $\Phi(w)\leq\tilde{c}$ and conclude that $\Phi(w)=\tilde{c}$. We are done.
$\hfill\Box$

\section{Radial Symmetry}\label{Symmetry}
We commence this section presenting some basic results about a modified Bessel kernel, defined for any $s>0$ by
\begin{equation}\label{Besselkernel}
g_s(x)=\frac{1}{(4\pi)^s\Gamma(s)}\int_{0}^\infty e^{-\pi|x|^2/\delta}e^{-m^2\delta/(4\pi)}\delta^{(2s-N)/2}\frac{\dd\delta}{\delta}.
\end{equation}

The results follow simply by adapting the proofs presented in \cite{Mizuta} or \cite{Stein}.
\begin{proposition}\label{prop1}For every $s>0$ we have
\begin{enumerate}
\item [$(i)$] $g_s\in L^1(\mathbb{R}^N)$;
\item [$(ii)$] $\hat{g}_s(\xi)=(m^2+4\pi^2|\xi|^2)^{-s}$. 
\end{enumerate}
\end{proposition}

The proof of Proposition \ref{prop1} is a consequence of the identity 
\[\int_{ \mathbb{R}^N}e^{-\pi|x|^2/\delta}\dd x=\delta^{N/2},\]
and the application of Fubini's theorem.

The next result follows immediately by considering the Fourier transform of $g_{s_1+s_2}$, applying Proposition \ref{prop1} and then the inversion formula. 
\begin{corollary}\label{cor}
For every $s_1,s_2>0$ it holds
\[g_{s_1}*g_{s_2}=g_{s_1+s_2}.\]
\end{corollary}
\begin{definition}
For a given $f\in L^p(\mathbb{R}^N)$ with $1\leq p\leq \infty$, we define the Bessel potential $I_s(f)$ by
\[I_s(f)=\left\{\begin{array}{ll}
g_s*f, &\text{if }\, s>0\\
f, &\text{if }\, s=0\end{array}\right.\]
\end{definition}

The next result is a consequence of Corollary \ref{cor}:
\begin{proposition}We have
\begin{enumerate}
\item [$(i)$] For any $f\in L^p(\mathbb{R}^N)$ with $1\leq p\leq\infty$,
\[I_s(f)\in L^p(\mathbb{R}^N)\quad\text{and}\quad |I_s(f)|_{p}\leq \frac{1}{m^{2s}}|f|_p;\]
\item [$(ii)$] $I_{s_1}\circ I_{s_2}=I_{s_1+s_2}$.
\end{enumerate}
\end{proposition}

\begin{definition}
For any $s>0$ and $1\leq p\leq\infty$, we define
\[L^{s,p}(\mathbb{R}^N)=\{g_s*f\,:\,f\in L^p(\mathbb{R}^N)\}.\]
If $u=g_s*f\in L^p(\mathbb{R}^N)$, we also define
\[\|u\|_{s,p}=|f|_p.\]
\end{definition}

The space $L^{s,p}(\mathbb{R}^N)$ is Banach, see \cite{Mizuta}.

\begin{remark}
Since
\[\mathcal{F}\left((-\Delta+m^2)u\right)(\xi)=(m^2+4\pi^2|\xi|^2)^s\hat{u}(\xi),
\]
it follows from Proposition \ref{prop1} that
\begin{align*}\hat{u}(\xi)&=(m^2+4\pi2|\xi|^2)^{-s}\mathcal{F}\left((-\Delta+m^2)^su\right)(\xi)=\hat{g}_s(\xi)\mathcal{F}\left((-\Delta+m^2)^su\right)(\xi)\\
&=\mathcal{F}\left(g_s*(-\Delta+m^2)^su\right)(\xi),
\end{align*}
from what follows
\[u=g_s*(-\Delta+m^2)^su.\]
	
Therefore, $u$ solves
\[(-\Delta+m^2)^su=f(u)\]
if, and only if,
\[u=g_s*f(u).\]
\end{remark}

We now state a result proven in \cite[Theorem 9]{MaChen}:
\begin{theorem}\label{tMa}
Let $q>\max\{\beta,\frac{N(\beta-1)}{\alpha}\}$. If $f\in L^{q/\beta}(\mathbb{R}^N)$, then $I_s(f)\in L^q(\mathbb{R}^N)$. Moreover, we have the estimate
\[|I_s(f)|_q\leq C|f|_{q/\beta},\]
where $C=C(\alpha,\beta,N,q)$
\end{theorem}

Let us consider the problem 
\begin{equation}\label{Rad1}
(-\Delta + m^2)^{s} u = f(u) \ \ \ \mbox{in} \ \ \ \mathbb{R}^{N}.
\end{equation}
where $0<s<1$, $N>2s$, $m \in \mathbb{R}\setminus\{0\}$ and $f\colon [0,\infty) \rightarrow \mathbb{R}$ a continuous function that satisfies
\begin{enumerate}
\item[$(s_1)$] $f'(t) \geq 0 $ and $f''(t) \geq 0$ for all $t \in [0,\infty)$.
\item[$(s_2)$] For any given $\beta\in (1, 2^{*}_{s} - 1)$, there exists $q \in [2,2^{*}_{s}]$ with $q > \max\{\beta,\frac{N(\beta-1)}{2s}\}$ such that  $f'(w) \in L^{q/(\beta-1)}(\mathbb{R}^{N}), \ \ \forall w \in H^{s}(\mathbb{R}^{N})$.
\end{enumerate}

We give some examples of functions satisfying our hypotheses ($s_1$) and ($s_2$):
\begin{enumerate}
\item[$(1)$] For any $\alpha\in (1,2^{*}_s-1)$, the function $f(t) = t^{\alpha}$ clearly fulfill ($s_1$). Taking $\beta=\alpha$, then 
\[\int_{\mathbb{R}^{N}}|f'(w)|^{\frac{q}{\beta-1}}\dd x = \alpha^{\frac{q}{\beta-1}}\int_{\mathbb{R}^{N}}|w|^{q} \dd x  < \infty,\quad\forall\, q \in [2,2^{*}_{s}].
\]
	
\item[$(2)$] Also for  $f(t)=t^{\alpha} + t^{\gamma}$, where $\alpha, \gamma\in (1,2^{*}_{s}-1)$, condition ($s_1$) is verified. Furthermore, choosing $\beta = \max\{\alpha,\gamma\} \in (1, 2^{*}_{s}-1)$,since
\begin{align}\label{Rad2}
\int_{\mathbb{R}^{N}} |f'(w)|^{\frac{q}{\beta-1}} \dd x \leq C\int_{\mathbb{R}^{N}} \left(|w|^{q\left(\frac{\alpha-1}{\beta-1}\right)} + |w|^{q\left(\frac{\gamma-1}{\beta-1}\right)}\right)\dd x,
\end{align}
there exists $q \in \left( \max\{\beta, \frac{N(\beta-1)}{2s}\}, 2^{*}_{s}\right)$ such that $2<q\left(\frac{\alpha-1}{\beta-1}\right) < 2^{*}_{s}$ and $2<q\left(\frac{\alpha-1}{\gamma-1}\right) < 2^{*}_{s}$. Thus, if $w \in H^{s}(\mathbb{R})$, then \eqref{Rad2} and the Sobolev immersions imply that $f'(w) \in L^{q/(\beta-1)}(\mathbb{R}^{N})$.
\item[$(3)$] Consider $f(t)= t\ln(1+t)$, for $t \in [0,\infty)$. Since
\[f'(t) = \ln(1+t) + \frac{t}{1+t} \geq 0\quad\text{and}\quad f''(t) = \frac{1}{1+t} + \frac{1}{(1+t)^{2}} \geq 0
\]
for all $t \in [0,\infty)$, we have ($s_1$). Since $f'(t) \leq 2t$ if $t\geq 0$, we have $f'(w) \in L^{q/(\beta-1)}(\mathbb{R}^{N})$ for any $1< \beta < 2^{*}_{s}-1$ and $q > \max\{\beta,\frac{N(\beta-1)}{2s} \}$.
\end{enumerate}\vspace*{.2cm}

We now apply the moving planes technique in integral form to show that any positive solution of \eqref{Rad1} is radially symmetric. We start fixing some notation. 

For any $\lambda \in \mathbb{R}$, define
\begin{align*}
\Sigma_{\lambda} &= \{x= (x_1,x_2,...,x_N)\in\mathbb{R}^N\,:\, x_{1} \leq \lambda \}, \\
T_{\lambda} &= \{ x\in\mathbb{R}^N\,:\, x_1 = \lambda \},\\
x_\lambda&=(2\lambda - x_1,x_2,...,x_{N}),\ \ \text{if }\ x \in \Sigma_{\lambda},\\ 
u_{\lambda}(x)&= u(x_{\lambda}).
\end{align*} 

\begin{lemma}\label{Lema1}
For any positive solution $u(x)$ of \eqref{Rad1} we have
\begin{equation}\label{Rad3}
u(x)- u_{\lambda}(x) = \int_{\Sigma_{\lambda}}\left(g_{s}(x-y) - g_{s}(x_{\lambda} - y)\right)\left(f(u(y)) - f(u_{\lambda}(y))\right)\dd y.
\end{equation}
\end{lemma}
\begin{proof}
Since $u$ is a solution of \eqref{Rad1}, it holds
\[u(x) = \left[g_s*(-\Delta + m^{2})^{s}u\right](x) =\left(g_s*f(u)\right)(x)=\int_{\mathbb{R}^N}g_s(x-y)f(u(y))\dd y,\]
where $g_{s}$ is the modified Bessel kernel \eqref{Besselkernel}.
	
Thus, the change of variables $y \mapsto y_{\lambda}$ yields
\begin{align*}
u(x)&=\int_{\Sigma_{\lambda}}g_{s}(x-y)f(u(y))\dd y + \int_{\Sigma_{\lambda}^{c}}g_{s}(x-y)f(u(y))\dd y\\
&=\int_{\Sigma_{\lambda}}g_{s}(x-y)f(u(y))\dd y + \int_{\Sigma_{\lambda}^{c}}g_{s}(x-y_{\lambda})f(u(y_{\lambda}))\dd y \\
&=\int_{\Sigma_{\lambda}}\left(g_{s}(x-y)f(u(y)) + g_{s}(x_{\lambda}-y)f(u_{\lambda}(y))\right)\dd y
\end{align*}
the last equality being a consequence of the fact that  $g_{s}$ is radially symmetric and $|x_{\lambda}-y|=|x - y_{\lambda}|$. 
	
In the last equality, changing $x$ for $x_{\lambda}$, since $g_{s}$ is radial and $|x-y|=|x_{\lambda}-y_{\lambda}|$, we obtain
\begin{align*}
u(x)=\int_{\Sigma_{\lambda}}\left(g_{s}(x_{\lambda}-y)f(u(y)) + g_{s}(x_{\lambda}-y)f(u_{\lambda}(y))\right)\dd y
\end{align*}
and our proof is complete.
$\hfill\Box$\end{proof}


\textbf{Proof of Theorem \ref{tsymmetry}}: 

In \textbf{Step 1}, we show that, for any negative $\lambda$, if $|\lambda|$ is big enough, then 
\begin{equation}\label{Rad4}
u(x) < u_{\lambda}(x).
\end{equation}

Therefore, we can move the plane $T_{\lambda}$ in the $x_1$-axis from a neighborhood of $-\infty$ to the right until \eqref{Rad4} remains valid. 

In \textbf{Step 2} we will show that above translation of $T_\lambda$ is possible until we reach $x_1=0$. Thus, we conclude that $u(x)\leq u_{0}(x)$ for any $x\in \Sigma_{0}$. 

However, the same process of Steps 1 and 2 can be repeated, moving the plane $T_\lambda$ from a neighborhood of $+\infty$ to the left in the $x_1$-axis,  thus concluding that $u(x)\geq u_{0}(x)$ for any $x \in \mathbb{R}^{N} \setminus \Sigma_{0}$. 

Thus, we conclude that $u$ is symmetric in relation to the plane $T_0$. Since the direction $x_1$ was arbitrarily chosen, we conclude that $u$ is symmetric with respect to any axis. Changing coordinates, we conclude that $u$ is symmetric and decreasing with respect to the origin in any direction.

\textbf{Step 1}. For $\lambda$ negative enough we have 
\begin{equation}\label{Rad5}
u(x) \leq u_{\lambda}(x), \forall x\in \Sigma_{\lambda}.
\end{equation}

Let us denote
\begin{equation}
\Sigma_{\lambda}^{-} = \left\{x \in \Sigma_{\lambda}; u_{\lambda}(x) < u(x) \right\}.
\end{equation}

Lemma \ref{Lema1} yields
\begin{align*}
u(x) - u_{\lambda}(x) &=\int_{\Sigma_{\lambda}\setminus \Sigma_{\lambda}^{-}}\left(g_{s}(x-y) - g_{s}(x_{\lambda} - y)\right)\left(f(u(y)) - f(u_{\lambda}(y))\right)\dd y \\	&\qquad +\int_{\Sigma_{\lambda}^{-}}\left(g_{s}(x-y) - g_{s}(x_{\lambda} - y)\right)\left(f(u(y)) - f(u_{\lambda}(y))\right)\dd y.
\end{align*}

Since  $|x_{\lambda} - y|>|x-y|$, $g_{s}$ is both positive and radially decreasing, and $f$ satisfies $(f_1)$, we have $g_{s}(x_{\lambda}-y) \leq g_{s}(x-y)$ and $f(u(y)) \geq f(u_{\lambda}(y)), \forall y \in \Sigma \setminus \Sigma_{\lambda}^{-}$. Thus,
\begin{align}\label{Est2}
u(x) - u_{\lambda}(x) \leq \displaystyle\int_{\Sigma_{\lambda}^{-}}g_{s}(x-y)\left(f(u(y) -f(u_{\lambda}(y)\right) \dd y.	
\end{align}

It follows from the mean value theorem the existence of $\theta \in (0,1]$ such that, for all $y \in \Sigma_{\lambda}^{-}$ we have
\[f(u(y)) - f(u_{\lambda}(y)) = f'\left(u(y) + \theta[ u_{\lambda}(y)-u(y)]\right)\left( u(y) - u_{\lambda}(y)\right).\]

Since $(f_1)$ implies that $f'$ is increasing, \[f'\left(u(y) + \theta[ u_{\lambda}(y)-u(y)]\right) \leq f'(u_{\lambda}(y)).\]

Substituting this estimate in \eqref{Est2} yields 
\begin{align*}
u(x) - u_{\lambda}(x) \leq \displaystyle\int_{\Sigma_{\lambda}^{-}}g_{s}(x-y) f'(u_{\lambda}(y)) \left( u(y) - u_{\lambda}(y)\right)\dd y.
\end{align*}

Now, fix $\beta = \frac{4s}{N} + 1$ and consider $q >\max\{\frac{4s}{N} + 1, 2\} = \max\{\beta,\frac{N(\beta-1)}{2s} \}$ (as given by hypothesis $(f_2)$). According to Theorem \ref{tMa} we have
\begin{align}\label{Rad7}
|u - u_{\lambda}|_{L^{q}(\Sigma_{\lambda}^{-})} &\leq C\left|I_{s}\left(f'(u_{\lambda})\left(u-u_{\lambda}\right)\right) \right|_{L^{q}(\Sigma_{\lambda}^{-})} \nonumber \\
&\leq C \left|f'(u_{\lambda}) \left( u - u_{\lambda}\right) \right|_{L^{\frac{q}{\beta}}(\Sigma_{\lambda}^{-})}\nonumber \\ 
& \leq C \left| f'(u_{\lambda}) \right|_{L^{\frac{q}{\beta-1}}(\Sigma_{\lambda}^{-})} \left| u - u_{\lambda}\right|_{L^{q}(\Sigma_{\lambda}^{-})}\\ 
& \leq C \left|f'(u_{\lambda}) \right|_{L^{\frac{q}{\beta-1}}(\Sigma_{\lambda})} \left| u - u_{\lambda} \right|_{L^{q}(\Sigma_{\lambda}^{-})}\nonumber \\
&= C\left|f'(u) \right|_{L^{\frac{q}{\beta-1}}(\Sigma_{\lambda}^{c})} \left|u - u_{\lambda}\right|_{L^{q}(\Sigma_{\lambda}^{-})},\nonumber 
\end{align}
the last inequality being a consequence of the change of variables $x\mapsto x_{\lambda}$. 

Since $(f_2)$ implies that $f'(u) \in L^{\frac{q}{\beta-1}}(\mathbb{R}^{N})$, there exists $N_0 > 0$ big enough so that  
\begin{equation}\label{Rad8}
\lambda \leq -N_{0}\quad \Rightarrow\quad   C\left| f'(u) \right|_{L^{\frac{q}{\beta-1}}(\Sigma_{\lambda}^{c})} \leq \frac{1}{2}.
\end{equation}

Applying \eqref{Rad8} in  \eqref{Rad7} produces $\left| u - u_{\lambda} \right|_{L^{q}(\Sigma_{\lambda}^{-})} = 0$ for any $\lambda \leq - N_0$. Thus, $\Sigma_{\lambda}^{-}$ has null measure for any $\lambda$ negative enough. 

\textbf{Step 2}. Let us suppose that the plane $T_\lambda$ can be moved to the right until $\lambda_{0} < 0$. If there exists $x^{*} \in \Sigma_{\lambda_0}$ such that $u(x^{*}) = u_{\lambda_0}(x^{*})$, then it follows from Lemma \ref{Rad3} that 
\begin{align}\label{Est}
0 &= u(x^{*}) - u_{\lambda_0}(x^{*})\nonumber\\ &=\int_{\Sigma_{\lambda_0}} \left(g_{s}(x^{*}-y) - g_{s}(x^{*}_{\lambda_{0}} - y) \right)\left( f(u(y)) - f(u_{\lambda_{0}}(y))\right) \dd y.
\end{align}

Since $g_{s}$ is radially decreasing and $|x^{*} - y| > |x^{*}_{\lambda_{0}} -y|$ in $\Sigma_{\lambda_0}$, then $g_{s}(x^{*} - y) < g_{s}(x^{*}_{\lambda_{0}}-y)$, from what follows
\begin{equation}\label{Rad9}
f(u(y)) = f(u_{\lambda_0}(y)), \ \ \forall\,y \in \Sigma_{\lambda_0}.
\end{equation}

According $(f_1)$,  \eqref{Rad9} only happens if $u(y) = u_{\lambda_0}(y)$ for all $y \in \Sigma_{\lambda_0}$. In this case, \eqref{Est} yields
\begin{equation*}
u(x) \equiv u_{\lambda_{0}}(x) \equiv 0 \ \ \mbox{in} \ \ \Sigma_{\lambda_0}.
\end{equation*}

This implies that $u(x) \equiv 0$, contradicting the fact that $u$ is a positive solution. By Step 1 we conclude that $\Sigma_{\lambda_{0}}$ has null measure, thus yielding 
\begin{equation*}
u(x) < u_{\lambda_{0}}(x), \ \ \forall\, x \in \Sigma_{\lambda_{0}}.
\end{equation*}

We claim that 
\begin{equation}\label{Rad10}
\lambda_{0} = \sup \left\{\lambda ; \  u(x) \leq u_{\lambda}(x), \ \forall x \in \Sigma_{\lambda}\right\} = 0.
\end{equation}

Supposing the contrary, that is $\lambda_{0} < 0$, we show that the plane $T_{\lambda_{0}}$ can be moved to the right, contradicting the definition of $\lambda_{0}$. 

Since $f'(u) \in L^{q/(\beta-1)}(\mathbb{R}^{N})$ it follows that, for any $\epsilon>0$ small enough, there exists $R>0$ big enough so that 
\begin{equation*}
\int_{\mathbb{R}^{N}\setminus B_{R}(0)} |f'(u) |^{\frac{q}{\beta-1}} \dd x < \epsilon.
\end{equation*}

Applying Lusin's theorem, for any $\delta > 0$ there exists a closed set $F_{\delta}$ such that $(u-u_{\lambda_{0}})\big|_{F_{\delta}}$ is continuous, with $F_{\delta} \subset B_{R}(0) \cap \Sigma_{\lambda_{0}} = E$ and $\mu(E-F_{\delta}) < \delta$. 

Since $u(x) < u_{\lambda_{0}}(x)$ in the interior of  $\Sigma_{\lambda_{0}}$, we obtain that $u(x) < u_{\lambda_{0}}(x)$ in $F_{\delta}$. 

Choose $\epsilon_1 > 0$ small enough so that, for any  $\lambda \in [\lambda_0, \lambda_{0}+\epsilon_1)$,
\[u-u_{\lambda} < 0, \quad\forall\, x \in F_{\delta}.\]

It follows that  
\begin{equation*}
\Sigma_{\lambda}^{-} \subset M:= \left(\mathbb{R}^{N}\backslash B_{R}(0)\right) \cup \left(E\setminus F_{\delta}\right) \cup \left[\left(\Sigma_{\lambda} \setminus \Sigma_{\lambda}^{-}\right) \cap B_{R}(0)\right].
\end{equation*}

Now take $\epsilon,\delta$ and $\epsilon_{1}$ small enough  to that, 
\begin{equation*}
C\left| f'(u) \right|_{L^{\frac{q}{\beta-1}(M)}} \leq \frac{1}{2}.
\end{equation*} 

Thus, 
\begin{equation*}
\left| u-u_{\lambda}\right|_{L^{q}(\Sigma_{\lambda}^{-})} \leq C\left|f'(u) \right|_{L^{\frac{q}{\beta-1}(\Sigma_{\lambda}^{-})}} \left|u - u_{\lambda} \right|_{L^{q}(\Sigma_{\lambda}^{-})} \leq \frac{1}{2} \left| u - u_{\lambda} \right|_{L^{q}(\Sigma_{\lambda}^{-})}.
\end{equation*}

It follows that $\Sigma_{\lambda}^{-}$ has null measure and therefore $u(x) \leq u_{\lambda}(x)$ for any $x \in \Sigma_{\lambda}$, contradicting the definition of $\lambda_{0}$. Thus, claim \eqref{Rad10} is proved.
We are done.
$\hfill\Box$

\textbf{Acknowledgements:} H. Bueno dedicates this paper to Stephan Luckhaus on his 65th. birthday. All the authors thank Olimpio Miyagaki and Minbo Yang for useful conversations.

\end{document}